\documentclass[pdftex]{article} 

\usepackage{amsfonts}
\usepackage{amssymb}
\usepackage{amsmath}
\usepackage{amsthm}
\usepackage{tikz}

\usepackage{cite}

\usepackage{comment}
\usepackage{accents}

\usepackage{hyperref}

\makeatletter
\newcommand{\leqnomode}{\tagsleft@true\let\veqno\@@leqno}
\newcommand{\reqnomode}{\tagsleft@false\let\veqno\@@eqno}
\makeatother

\usepackage[shortlabels]{enumitem}
\setlist[enumerate]{leftmargin=*,align=left,labelindent=\parindent}

\makeatletter
\newcommand{\myitem}[1][]{%
\item[#1]\protected@edef\@currentlabel{#1}\ignorespaces%
}
\makeatother

\newtheorem{theorem}{Theorem}[section]
\newtheorem{proposition}[theorem]{Proposition}
\newtheorem{lemma}[theorem]{Lemma}
\newtheorem{corollary}[theorem]{Corollary}

\theoremstyle{definition}
\newtheorem{definition}[theorem]{Definition}
\theoremstyle{remark}
\newtheorem{remark}[theorem]{Remark}
\newtheorem*{notation*}{Notation}

\newcommand{\ha}{\mathsf{HA}}
\newcommand{\pa}{\mathsf{PA}}
\newcommand{\T}{{T}}
\newcommand{\FDNS}{\mathrm{DNS}}

\newcommand{\DNE}[1]{{#1}\text{-}\mathrm{DNE}}
\newcommand{\DNEC}[1]{{#1}\text{-}\undertilde{\mathrm{DNE}}}
\newcommand{\DNSC}[1]{{#1}\text{-}\undertilde{\mathrm{DNS}}}
\newcommand{\LEM}[1]{{#1}\text{-}\mathrm{LEM}}
\newcommand{\CD}[1]{{#1}\text{-}\mathrm{CD}}

\newcommand{\DNS}[1]{{#1}\text{-}\mathrm{DNS}}

\newcommand{\DNER}[1]{{#1}\text{-}\mathrm{DNE}\text{-}\mathrm{R}}
\newcommand{\CDR}[1]{{#1}\text{-}\mathrm{CD}\text{-}\mathrm{R}}
\newcommand{\DMLR}[1]{{#1}\text{-}\mathrm{DML}\text{-}\mathrm{R}}
\newcommand{\DMLDR}[1]{{#1}\text{-}\mathrm{DML^{\perp}}\text{-}\mathrm{R}}

\newcommand{\Prf}[2]{\mathrm{Pf}\left({#1} , \left\ulcorner {#2} \right\urcorner\right)}
\newcommand{\ph}{\$}
\newcommand{\PH}{*}

\newcommand{\F}{\mathrm{F}}
\newcommand{\B}{\mathrm{B}}
\newcommand{\U}{\mathrm{U}}
\newcommand{\E}{\mathrm{E}}
\newcommand{\FV}[1]{\mathrm{ FV} \left({#1}\right)}

\newcommand{\R}{\mathcal{R}}
\newcommand{\J}{\mathcal{J}}
\newcommand{\Q}{\mathcal{Q}}

\newcommand{\V}{\mathcal{V}}
\newcommand{\EP}{\mathrm{E}\Pi}
\newcommand{\ES}{\mathrm{E}\Sigma}

\newcommand\ang[1]{{\langle #1 \rangle} }

\newcommand{\lr}{\leftrightarrow}
\newcommand{\LR}{\Leftrightarrow}
\newcommand{\llr}{\longleftrightarrow}
\newcommand{\LLR}{\Longleftrightarrow}
\newcommand{\lra}{\longrightarrow}
\newcommand{\vp}{\varphi}

\newcommand{\QF}[1]{{#1}_{\mathrm{qf}}}

\newcommand{\wt}[1]{\widetilde{#1}}
\newcommand{\ut}[1]{\undertilde{#1}}
\newcommand{\degree}[1]{\mathit{deg}(#1)}
\newcommand{\alt}[1]{\mathit{Alt}(#1)}
\DeclareSymbolFont{largesymbol}{OMX}{yhex}{m}{n}
\DeclareMathAccent{\Widehat}{\mathord}{largesymbol}{"62}

\title{Conservation theorems on semi-classical arithmetic}

\author{Makoto Fujiwara\footnote{Email: makotofujiwara@meiji.ac.jp}
\footnote{School of Science and Technology, Meiji University, 1-1-1 Higashi-Mita, Tama-ku, Kawasaki-shi, Kanagawa 214-8571, Japan.}
and Taishi Kurahashi\footnote{Email: kurahashi@people.kobe-u.ac.jp}
\footnote{Graduate School of System Informatics,
Kobe University,
1-1 Rokkodai, Nada, Kobe 657-8501, Japan.}}
\date{\today}

\begin{document}
\maketitle

\begin{abstract}
We systematically study conservation theorems on theories of semi-classical arithmetic, which lie in-between classical arithmetic $\pa$ and intuitionistic arithmetic $\ha$.
Using a generalized negative translation, we first provide a new structured proof of the fact that $\pa$ is $\Pi_{k+2}$-conservative over  $\ha + \LEM{\Sigma_k}$ where $\LEM{\Sigma_k}$ is the axiom scheme of the law-of-excluded-middle restricted to formulas in $\Sigma_k$.
In addition, we show that this conservation theorem is optimal in the sense that for any semi-classical arithmetic $T$, if $\pa$ is $\Pi_{k+2}$-conservative over $T$,  then $\T$ proves $\LEM{\Sigma_k}$.
In the same manner, we also characterize conservation theorems for other well-studied classes of formulas by fragments of classical axioms or rules.
This reveals the entire structure of conservation theorems with respect to the arithmetical hierarchy of classical principles.
\end{abstract}

\section{Introduction}
It is well-known that classical first-order arithmetic $\pa$ is $\Pi_2$-conservative over  intuitionistic first-order arithmetic $\ha$.
There are several approaches to prove this fundamental fact.
One simple and well-known approach is to apply the negative (or double negation) translation followed by the Friedman A-translation \cite{Fri78}.
Another possible approach is to apply a generalized negative translation developed systematically by Ishihara \cite{Ishi00, Ishi12}.
In fact, the latter is a combination of Gentzen's negative translation and the Friedman A-translation (cf. \cite[Section 4]{Ishi12}).
In \cite[Theorem 6.14]{FK20-1}, the authors showed a conservation result which generalizes the aforementioned conservation result on $\pa$ and $\ha$ in the context of semi-classical arithmetic (which lies between classical and intuitionistic arithmetic).
In fact, the following is an immediate corollary of \cite[Theorem 6.14]{FK20-1}:
\begin{proposition}
\label{prop: PA is Pi_k+2-cons. over HA + Sigma_k-LEM}
$\pa$ is $\Pi_{k+2}$-conservative over $\ha +\LEM{\Sigma_{k}}$ where $\LEM{\Sigma_{k}}$ is the axiom scheme of the law-of-excluded-middle restricted to formulas in $\Sigma_k$.
\end{proposition}
The proof of \cite[Theorem 6.14]{FK20-1} in that paper is similar to the former approach in the sense of using the Friedman A-translation.
However, the proof has somewhat intricate structure in dealing with the Friedman A-translation of the inner part of Kuroda's negative translation.
In Section \ref{sec: Revativised Ishihara's result}, by extending the latter approach from \cite{Ishi00, Ishi12} in the context of semi-classical arithmetic, we provide a much more structured proof of \cite[Theorem 6.14]{FK20-1}.
As an advantage of the structured proof, we obtain an extended conservation result for much larger classes of formulas (see Theorem \ref{thm: conservativity on semi-classical arithmetic} and Remark \ref{rem: our conservation result is an extension of the previous one}).

In Section \ref{sec: formula-classes}, we relate the classes used in Section \ref{sec: Revativised Ishihara's result} (which are based on the classes introduced in \cite{Ishi00}) to the classes $\U_k$ and $\E_k$ introduced in Akama et. al \cite{ABHK04} for studying the hierarchy of the constructively-meaningful fragments of classical axioms (including the law-of-excluded-middle and the double-negation-elimination).
The classes $\E_k$ and $\U_k$ correspond to classical $\Sigma_k$ and $\Pi_k$ respectively in the sense that every formula in $\E_k$ (resp. $\U_k$) is equivalent over $\pa$ to some formula in $\Sigma_k$ (resp. $\Pi_k$) and vice versa.
This investigation reveals that our extended conservation theorem for $\ha+\LEM{\Sigma_k}$ is applicable to all formulas in $\E_{k+1}$ (see Corollary \ref{cor: PA is Ek+1-cons. over HA+SkLEM}).

In Sections \ref{sec: CONS for classes of formulas}, \ref{sec: CONS for classes of sentences} and \ref{sec: Summary}, we investigate the entire structure of conservation theorems in the arithmetical hierarchy of classical principles which was systematically studied first in Akama et. al \cite{ABHK04} and further extended by the authors recently in \cite{FK20-2}.
The first motivation of this investigation comes from the observation that for any semi-classical arithmetic $\T$ such that $\pa$ is $\Pi_{k+2}$-conservative over $\T$, $T$ proves $\LEM{\Sigma_{k}}$ (cf. Lemma \ref{lem: SkvPik-cons->Sk-LEM}).
This means that Proposition \ref{prop: PA is Pi_k+2-cons. over HA + Sigma_k-LEM} is optimal in the sense that one cannot replace $\ha+\LEM{\Sigma_k}$ by any semi-classical arithmetic which does not prove $\LEM{\Sigma_k}$.
Another motivating fact is that for any semi-classical arithmetic $\T$, $\pa$ is $\Pi_{2}$-conservative over $\T$ if and only if $\T$ is closed under Markov's rule for primitive recursive predicate (cf. \cite[Section 3.5.1]{ConstMathI}).
Thus the $\Pi_{2}$-conservativity is also characterized by the $\Sigma_1$-fragment of the double-negation-elimination rule.
Then it is natural to ask whether this can be relativized in the context of semi-classical arithmetic. 
Motivated by these facts, in Sections \ref{sec: CONS for classes of formulas} and \ref{sec: CONS for classes of sentences}, we study the conservation theorems for the well-studied classes (including $\Pi_k$, $\Sigma_k$, the classes in \cite{ABHK04} and their closed variants) and characterize them by fragments of classical axioms or rules.
The conservativity for a class of formulas is equivalent to that restricted only to sentences if the class is closed under taking a universal closure.
Then the strength of the conservativity for e.g. $\Pi_k$ does not vary even if we restrict them only to sentences.
On the other hand, since $\Sigma_{k}$ etc. are not closed under taking a universal closure, this is not the case for such classes.
We investigate the conservation theorems for classes of formulas in Section \ref{sec: CONS for classes of formulas} and those for sentences in Section \ref{sec: CONS for classes of sentences}.
Through a lot of delicate arguments in semi-classical arithmetic, we reveal the detailed structure consisting of the conservation theorems and some fragment of logical principles, which are summarized in Section \ref{sec: Summary}.
This exhaustive investigation shed light on the close connection between the notion of conservativity and classical axioms and rules in semi-classical arithmetic.
For the purpose of future use, we present our characterization results in a generalized form with adding a set $X$ of sentences into the theories in question.

In the end of this paper, as an appendix, we show the relativized soundness theorem of the Friedman A-translation for $\ha +\LEM{\Sigma_k}$.
By this relativized soundness theorem, one may obtain a simple proof of Proposition \ref{prop: PA is Pi_k+2-cons. over HA + Sigma_k-LEM} just by imitating the aforementioned Friedman's approach.

\section{Framework}
We work with a standard formulation of intuitionistic arithmetic $\ha$ described e.g.~in \cite[Section 1.3]{Tro73}, which has function symbols for all primitive recursive functions.
Our language contains all the logical constants $\forall, \exists, \to, \land, \lor$ and $ \perp$.
In our proofs, when we use some principle (including induction hypothesis [I.H.]) which is not available in $\ha$, it will be exhibited explicitly.
As regards basic reasoning over intuitionistic first-order logic, we refer the reader to \cite[Section 6.2]{vD13}.

Throughout this paper, let $k$ be a natural number (possibly $0$).
The classes $\Sigma_k$ and $\Pi_k$ of $\ha$-formulas are defined as follows:
\begin{itemize}
	\item Let $\Sigma_0 $ and $ \Pi_0$ be the set of all quantifier-free formulas; 
	\item $\Sigma_{k+1} : = \{\exists x_1, \dots,  x_n \, \varphi \mid \varphi \in \Pi_k\}$;
	\item $\Pi_{k+1} : = \{\forall x_1, \dots, x_n\, \varphi \mid \varphi \in \Sigma_k\}$.
\end{itemize}
Let $\FV{\varphi}$ denote the set of all free variables in $\varphi$. 
Note that every formula $\varphi$ in $\Sigma_{k+1}$ (resp.~$\Pi_{k+1}$) is equivalent over $\ha$ to some formula $\psi$ in $\Sigma_{k+1}$ (resp.~$\Pi_{k+1}$) such that $\FV{\varphi} = \FV{\psi}$ and $\psi$ is of the form $\exists x \psi'$ (resp.~$\forall x \psi'$) where $\psi'$ is $\Pi_k$ (resp.~$\Sigma_k$).
For convenience, we assume that $\Sigma_m$ and $\Pi_m$ denote the empty set for negative integers $m$.

The classical variant $\pa$ of $\ha$ is defined as $\ha + {\rm LEM}$ or $\ha + {\rm DNE}$, where ${\rm LEM}$ is the axiom scheme of the law-of-excluded-middle ${\vp \lor \neg \vp}$ and  ${\rm DNE}$ is that of the double-negation-elimination ${\neg \neg \vp \to \vp}$.
Recall that $\LEM{\Sigma_k}$ and $\DNE{\Sigma_k}$ are ${\rm LEM}$ and ${\rm DNE}$ restricted to formulas in $\Sigma_k$ (possibly containing free variables) respectively.
Similarly, $\LEM{\Pi_k}$ and $\DNE{\Pi_k}$ are defined for $\Pi_k$.
We call a theory $\T$ such that $\ha \subseteq \T \subseteq \pa$ {\bf semi-classical arithmetic}.

Unless otherwise stated, the inclusion
between classes
of $\ha$-formulas is to be understood modulo equivalences over $\ha$.
That is, for classes $\Gamma$ and $\Gamma'$ of $\ha$-formulas, $\Gamma \subseteq \Gamma'$ denotes that for all $\vp\in \Gamma$, there exists $\vp'\in \Gamma'$ such that $\FV{\vp}=\FV{\vp'}$ and $\ha \vdash \vp'\lr \vp$, and $\Gamma = \Gamma'$ denotes $\Gamma \subseteq \Gamma'$ and  $\Gamma' \subseteq \Gamma$.
In this sense, one may think of $\Sigma_k$ and $\Pi_k$ as sub-classes of $\Sigma_{k'}$ and $\Pi_{k'}$ for all $k'>k$ (see \cite[Remark 2.5]{FK20-1}).

\section{A relativization of Ishihara's conservation result in semi-classical arithmetic}
\label{sec: Revativised Ishihara's result}
In this section, we simulate Ishihara's proof of \cite[Theorem 10]{Ishi00} in the specific context of semi-classical arithmetic studied in \cite{ABHK04, FK20-1} with some additional arguments.
We first recall the translation studied in \cite{Ishi00}.
In the context of the translation, without otherwise stated,  we work in the language with an additional predicate symbol ${\ph}$ of arity $0$, which behaves as ``place holder'' (See \cite{Ishi00, Ishi12} for more information).
Let $\ha^{\ph}$ denote $\ha$ in that language.
On the other hand, $\ha^{\ph} + \LEM{\Sigma_k}$ denotes $\ha^{\ph}$ augmented with
$\LEM{\Sigma_k}$ for ``$\ha$''-formulas.
\begin{definition}[cf. {\cite[Definition 3]{Ishi00}}]
Let $\neg_{\ph} \vp$ denote $\vp \to {\ph}$.
For each formula $\vp$, its ${\ph}$-translation $\vp^{\ph}$ is defined inductively by the following clauses:
\begin{itemize}
    \item 
For $P$ prime such that $P \not \equiv \, \perp$, $P^{\ph}:\equiv \neg_{\ph} \neg_{\ph} P$;
\item
$\perp^{\ph} :\equiv {\ph}$;
\item
$(\vp_1 \circ \vp_2)^{\ph} :\equiv \vp_1^{\ph} \circ \vp_2^{\ph} $\,  for $\circ \in \{ \land, \to\}$;
\item
$(\vp_1 \lor \vp_2)^{\ph} :\equiv \neg_{\ph} \neg_{\ph}\left( \vp_1^{\ph} \lor \vp_2^{\ph}\right) $;
\item
$\left( \forall x \vp \right)^{\ph} :\equiv \forall x \vp^{\ph}$;
\item
$\left( \exists x \vp \right)^{\ph} :\equiv \neg_{\ph} \neg_{\ph} \exists x \vp^{\ph}$.
\end{itemize}
It is straightforward to see $\FV{\vp} = \FV{\vp^{\ph}}$.
\end{definition}

\begin{proposition}[cf. {\cite[Proposition 4]{Ishi00} and \cite[Section 4]{Ishi12}}]
\label{Soundness of *-translation}
\noindent
\begin{enumerate}
    \item 
\label{item: HA |- N*N*A* <-> A*}
For any $\ha$-formula $\vp$, $\ha^{\ph}  \vdash \neg_{\ph}\neg_{\ph} \vp^{\ph} \lr \vp^{\ph} ;$
\item
\label{item: PA |- A => HA |- A*}
For any $\ha$-formula $\vp$ and any set $X$ of $\ha$-sentences, if $\pa +X \vdash \vp$, then $\ha^{\ph} +X^{\ph} \vdash \vp^{\ph}$, where $X^{\ph}:= \{\psi^{\ph} \mid \psi \in X\}$.
\end{enumerate}
\end{proposition}
\begin{proof}
The proofs are routine:
One can show \eqref{item: HA |- N*N*A* <-> A*} by induction on the structure of formulas, and \eqref{item: PA |- A => HA |- A*}
by induction on the length of the proof of $\vp$ in $\pa +X$.
\end{proof}

\begin{corollary}
\label{cor: PA |- A <-> B => HA |- A* <-> B*}
For any $\ha$-formulas $\vp_1$ and $\vp_2$, if $\pa \vdash \vp_1 \lr \vp_2$, then $\ha^{\ph} \vdash \vp_1^{\ph} \lr \vp_2^{\ph}$.
\end{corollary}
\begin{proof}
If $\pa $ proves $\vp_1 \lr \vp_2$, by Proposition \ref{Soundness of *-translation}.\eqref{item: PA |- A => HA |- A*}, we have that $\ha^{\ph}$ proves $\left( \vp_1 \lr \vp_2\right)^{\ph}$, which is in fact $\vp_1^{\ph} \lr \vp_2^{\ph}$.
\end{proof}

\begin{lemma}
\label{lem: HA|- Aqf* <-> Aqf v * }
For a quantifier-free formula $\QF{\vp}$ of $\ha$,
$\ha^{\ph} $ proves $\QF{\vp}^{\ph} \lr \QF{\vp} \lor {\ph}$.
\end{lemma}
\begin{proof}
By induction on the structure of quantifier-free formulas of $\ha$.
   
    The case of $ \perp$:
    Since $\perp^{\ph} \equiv {\ph}$, we have trivially $\ha^{\ph} \vdash \perp^{\ph} \lr \perp \lor {\ph}$.
    
The case of that $\QF{\vp}$ is a prime formula but $\perp$:
It is trivial that $\ha^{\ph}$ proves $\QF{\vp} \lor {\ph} \to  \neg_{\ph} \neg_{\ph} \QF{\vp}$.
On the other hand, since $\ha^{\ph} $ proves  $\QF{\vp} \lor \neg \QF{\vp}$ and $\neg_{\ph} \neg_{\ph} \QF{\vp} \land \neg \QF{\vp} \to {\ph}$, we also have that $\ha^{\ph}$ proves $\neg_{\ph} \neg_{\ph} \QF{\vp} \to \QF{\vp} \lor {\ph}$.

The case of $\QF{\vp}\equiv \vp_1 \land \vp_2$:
We have that $\ha^{\ph}$ proves
$$
(\vp_1 \land \vp_2)^{\ph} \lr \vp_1^{\ph} \land \vp_2^{\ph} \underset{\text{I.H.}}{\llr} (\vp_1 \lor {\ph}) \land (\vp_2 \lor {\ph}) \lr (\vp_1 \land \vp_2)\lor {\ph}.
$$

The case of $\QF{\vp}\equiv \vp_1 \lor \vp_2$:
Since $\vp_1$ and $\vp_2$ are decidable in $\ha$ (note that they are quantifier-free), we have that $\ha^{\ph}$ proves $(\vp_1 \lor \vp_2) \lor (\neg \vp_1 \land \neg \vp_2)$.
In the latter case of the disjunction, we have $\neg_{\ph}(\vp_1 \lor \vp_2 \lor {\ph})$.
Thus  $\ha^{\ph}$ proves
$$ \neg_{\ph} \neg_{\ph}(\vp_1 \lor \vp_2 \lor {\ph}) \to (\vp_1 \lor \vp_2) \lor {\ph} .$$
On the other hand, $\ha^{\ph}$ also proves
$$
(\vp_1 \lor \vp_2) \lor {\ph} \to \neg_{\ph} \neg_{\ph} (\vp_1 \lor \vp_2 \lor {\ph})
.
$$
Thus $\ha^{\ph}$ proves
$$
(\vp_1 \lor \vp_2)^{\ph} \equiv \neg_{\ph} \neg_{\ph} (\vp_1^{\ph} \lor \vp_2^{\ph})  \underset{\text{I.H.}}{\llr}  \neg_{\ph} \neg_{\ph}(\vp_1 \lor \vp_2 \lor {\ph}) \leftrightarrow (\vp_1 \lor \vp_2) \lor {\ph} .
$$

The case of $\QF{\vp}\equiv \vp_1 \to \vp_2$:
Assume $\vp_1 \lor {\ph} \to  \vp_2 \lor {\ph} $.
Then we have
\begin{equation}
    \label{eq: vp1 -> vp1 v *}
\vp_1 \to  \vp_2 \lor {\ph} .
\end{equation}
Since $\vp_1$ and $\vp_2$ are decidable in $\ha$ (note that they are quantifier-free), we have that $\ha^{\ph}$ proves $(\vp_2 \lor \neg \vp_1) \lor (\vp_1 \land \neg \vp_2)$.
In the former case, we have $\vp_1 \to \vp_2$.
In the latter case, by \eqref{eq: vp1 -> vp1 v *}, we have ${\ph}$.
Thus $\ha^{\ph}$ proves $$
\left(\vp_1 \lor {\ph} \to  \vp_2 \lor {\ph} \right) \to  \left( \vp_1\to \vp_2 \right) \lor {\ph}.$$
On the other hand, $\ha^{\ph}$ also proves
$$
  \left( \vp_1\to \vp_2 \right) \lor {\ph}
  \to \left(\vp_1 \lor {\ph} \to  \vp_2 \lor {\ph} \right)
  .$$
Thus $\ha^{\ph}$ proves
$$
(\vp_1 \to \vp_1) \lor {\ph} \leftrightarrow  \left( \vp_1 \lor {\ph} \to  \vp_2 \lor {\ph} \right)
 \underset{\text{I.H.}}{\llr}
\left(  \vp_1^{\ph} \to \vp_2^{\ph} \right) \equiv \left( \vp_1 \to \vp_2\right)^{\ph}.$$
\end{proof}

The following lemma is the key for our generalized conservation result:
\begin{lemma}
\label{Kurahashi Lemma: A* <-> A v *}
For a formula $\vp$ of $\ha$,
the following hold$:$
\begin{enumerate}
    \item
    \label{item: A v * for Pi_{k+1}}
    If $\vp\in \Pi_{k}$, $\ha^{\ph} + \LEM{\Sigma_{k}} \vdash \vp^{\ph} \lr \vp \lor {\ph} ;$
    \item
    \label{item: A v * for Sigma_{k+1}}
    If $\vp\in \Sigma_{k}$, $\ha^{\ph} + \LEM{\Sigma_{k}} \vdash \vp^{\ph} \lr \vp \lor {\ph}.$
\end{enumerate}
Note that $\LEM{\Sigma_k}$ is an axiom scheme in the language of $\ha$ (which does not contain $\ph$).
\end{lemma}
\begin{proof}
By simultaneous induction on $k$.
The base case is by Lemma \ref{lem: HA|- Aqf* <-> Aqf v * }.
Assume items \ref{item: A v * for Pi_{k+1}} and \ref{item: A v * for Sigma_{k+1}} for $k$ to show those for $k+1$.
First, for the first item, let $\vp:\equiv \forall x \vp_1$ where $\vp_1 \in \Sigma_k$.
By the induction hypothesis, we have $\ha^{\ph} + \LEM{\Sigma_k} \vdash  \vp_1^{\ph} \leftrightarrow \vp_1 \lor \ph $.
Note that $\ha^{\ph}$ proves $\forall x \vp_1 \lor \ph  \to \forall x (\vp_1 \lor \ph) $.
In the following, we show the converse $\forall x (\vp_1 \lor \ph) 
\to \forall x \vp_1 \lor \ph $ inside $\ha^{\ph} + \LEM{\Sigma_{k+1}}$.
Since $\neg \vp_1$ has some equivalent formula in $\Pi_k$ in the presence of $\DNE{\Sigma_{k-1}}$ (cf. Remark \ref{rem: duals} below), by $\LEM{\Sigma_{k+1}}$, we have now $\exists x \neg \vp_1 \lor \neg \exists x \neg \vp_1$.
In the former case, we have $\ph$ by using our assumption $\forall x (\vp_1 \lor \ph) $.
In the latter case, we have $\forall x \vp_1$ since $\neg \exists x \neg \vp_1 \leftrightarrow \forall x \neg \neg \vp_1$ and $\LEM{\Sigma_{k+1}}$ implies $\DNE{\Sigma_{k+1}}$.
Thus  $\ha^{\ph} + \LEM{\Sigma_{k+1}} $ proves $\forall x (\vp_1 \lor \ph) 
\to \forall x \vp_1 \lor \ph $.
Then we have that $\ha^{\ph} + \LEM{\Sigma_{k+1}}$ proves
$$
\vp^\ph \equiv 
\forall x \vp_1^{\ph} \underset{\text{[I.H.] }\LEM{\Sigma_{k}}}{\llr} 
\forall x (\vp_1 \lor \ph)
\underset{\LEM{\Sigma_{k+1}}}{\llr} 
\forall x \vp_1 \lor \ph .
$$

Next, for the second item, let $\vp:\equiv \exists x \vp_1$ where $\vp_1 \in \Pi_k$.
Note that $\vp^{\ph}$ is $\neg_{\ph} \neg_{\ph} \exists x \vp_1^{\ph}$.
By the induction hypothesis, we have $\ha^{\ph} + \LEM{\Sigma_k}$ proves $\vp_1^{\ph} \leftrightarrow \vp_1 \lor \ph $, and hence, $ \vp^{\ph} \leftrightarrow \neg_{\ph} \neg_{\ph} \exists x \vp_1$.
Then it is trivial that $\ha^{\ph} + \LEM{\Sigma_k} $ proves $\exists x \vp_1 \lor \ph \to \vp^{\ph}$.
In the following, we show the converse direction inside $\ha^{\ph} + \LEM{\Sigma_{k+1}}$.
By $\LEM{\Sigma_{k+1}}$, we have now $
\exists x \vp_1 \lor \neg \exists x \vp_1 $.
Then it suffices to show $\neg \exists x \vp_1 \land \neg_{\ph} \neg_{\ph} \exists x \vp_1  \to  \ph$, which is trivial since $\neg \exists x \vp_1 \to \neg_{\ph} \exists x \vp_1$.
\end{proof}

\begin{corollary}
\label{cor: A v * for Sigma_k}
For
a formula $\vp$ of $\ha$, if $\vp \equiv \exists x \vp_1 $ with $\vp_1 \in \Pi_{k}$, then $\ha^{\ph} + \LEM{\Sigma_{k}} \vdash \exists x \left({\vp_1}^{\ph} \right) \lr \vp \lor {\ph}.$
\end{corollary}
\begin{proof}
Since $\exists x \vp_1 \lor \ph \leftrightarrow \exists x (\vp_1 \lor \ph)$, the corollary is trivial by Lemma \ref{Kurahashi Lemma: A* <-> A v *}.\eqref{item: A v * for Pi_{k+1}}.
\end{proof}

In the context of intuitionistic/semi-classical arithmetic, a formula does not have an equivalent formula of the prenex normal form (namely, formula in $\Sigma_k$ or $\Pi_k$) while it does in classical arithmetic.
Because of this fact, the conservation theorem only for prenex formulas is not applicable in many practical cases.
On the other hand, Akama et. al \cite{ABHK04} introduced the classes $\U_k $ and $\E_k$ of formulas which correspond to classical $\Pi_k$ and $\Sigma_k$ respectively in the sense that every formula in $\U_k$ (resp. $\E_k$) is equivalent over $\pa$ to some formula in $\Pi_k$ (resp. $\Sigma_k$) and vice versa.
In addition, the authors introduced in \cite{FK20-1} the classes $\U_k^+ $ and $\E_k^+$, which are cumulative versions of $\U_k $ and $\E_k$.
For obtaining the conservation results for the classes as large as possible, 
we introduce classes $\R_k$ and $\J_k$ (see Definition \ref{def: Rk and Jk}), which relativize $\R$ and $\J$ in \cite{Ishi00} respectively with regard to the formulas of degree $\leq k$ in the sense of \cite{ABHK04, FK20-1}.

To make the definitions absolutely clear, we recall some notions in \cite{ABHK04, FK20-1}:
An {\bf alternation path} is a finite sequence of $+$ and $-$ in which $+$ and $-$ appear alternatively.
For an alternation path $s$, let $i(s)$ denote the first symbol of $s$ if $s \not \equiv \ang{\, }$ (empty sequence); $ \times$ if $s \equiv \ang{\, }$.
Let $s^{\perp}$ denote the alternation path which is obtained by switching $+$ and $-$ in $s$, and let $l(s) $ denote the length of $s$.
For a formula $\vp$, the set of alternation paths $\alt{\vp}$ of $\vp$ is defined as follows:
\begin{itemize}
    \item 
    If $\vp$ is quantifier-free, then $\alt{\vp} := \{ \ang{\, } \}$;
    \item
    Otherwise, $\alt{\vp}$ is defined inductively by the following clauses:
    \begin{itemize}
        \item
        If $\vp \equiv \vp_1 \land \vp_2$ or $\vp \equiv \vp_1 \lor \vp_2$, then $\alt{\vp} := \alt{\vp_1} \cup \alt{\vp_2}$;
        \item
        If $\vp \equiv \vp_1 \to \vp_2$, then $\alt{\vp} := \{ s^{\perp} \mid s \in \alt{\vp_1}\} \cup \alt{\vp_2}$;
\item
If $\vp \equiv \forall x \vp_1 $, then $\alt{\vp} :=\{s \mid s\in \alt{\vp_1} \text{ and } i(s)\equiv -\} \cup \{-s \mid s\in \alt{\vp_1} \text{ and } i(s)\not \equiv - \} $;
\item
If $\vp \equiv \exists x \vp_1 $, then $\alt{\vp} :=\{s \mid s\in \alt{\vp_1} \text{ and } i(s)\equiv + \} \cup \{+s \mid s\in \alt{\vp_1} \text{ and } i(s)\not \equiv + \} $.
    \end{itemize}
\end{itemize}
In addition, for a formula $\vp$, the degree $\degree{\vp}$ of $\vp$ is defined as 
$$\degree{\vp} := \max \{l(s) \mid s \in \alt{\vp}  \} .$$

\begin{definition}[cf. {\cite[Definition 2.4]{ABHK04} and \cite[Definition 2.11]{FK20-1}}]
\label{def: Classes}
The classes $\F_k, \U_k, \E_k $, $\F_k^+$, $\U_k^+ $ and $ \E_k^+ $ of $\ha$-formulas are defined as follows:
\begin{itemize}
    \item
    $\F_k := \{ \vp \mid \degree{\vp}=k  \}    ;\,\, \F_k^+ := \{ \vp \mid \degree{\vp}\leq k  \}   ;$
    \item
    $\U_0:=\E_0:=\F_0 \, (=\Sigma_0 =\Pi_0)$;
    \item
$\U_{k+1} := \{ \vp \in \F_{k+1} \mid i(s) \equiv - \text{ for all }s\in \alt{\vp} \text{ such that }l(s) =k+1 \}$;
\item
$\E_{k+1} := \{ \vp \in \F_{k+1} \mid i(s) \equiv + \text{ for all }s\in \alt{\vp} \text{ such that }l(s) =k+1  \}$;
    \item
$\begin{displaystyle}
\U_k^+ := \U_k \cup \bigcup_{i<k} \F_i;\, \, \E_k^+ := \E_k \cup \bigcup_{i<k} \F_i .
\end{displaystyle}
$ 
\end{itemize}
\end{definition}

\begin{remark}
\label{rem: On the non-commutativity of Ek, Uk and Fk}
As shown in \cite[Proposition 4.6]{FK20-1}, for any $\vp \in \U_k^+$ and $\psi\in \E_k^+$, there exist $\vp'\in \U_k$ and $\psi' \in \E_k$ such that $\FV{\vp}=\FV{\vp'}$, $\FV{\psi}=\FV{\psi'}$, $\ha \vdash \vp \lr \vp'$ and $\ha \vdash \psi \lr \psi'$.
Then it also follows that for any $\vp\in \F_k^+$, there exists $\vp'\in \F_k$ such that $\FV{\vp}=\FV{\vp'}$ and $\ha \vdash \vp \lr \vp'$.
Thus one may identify $\E_k^+, \U_k^+$ and $\F_k^+$ with $\E_k, \U_k$ and $\F_k$ respectively without loss of generality.
\end{remark}

Then the authors showed the following prenex normal form theorem:
\begin{theorem}[cf. {\cite[Theorem 5.3]{FK20-1} which corrects \cite[Theorem 2.7]{ABHK04}}]
\label{thm: PNFT}
For a $\ha$-formula $\varphi$,
the following hold$:$
\begin{enumerate}
        \item 
        \label{eq: item for Sigma_k in PNFT}
    If $\varphi \in \E_k^+$, then
        there exists $\varphi' \in \Sigma_k$  such that $\FV{\vp}=\FV{\vp'}$ and
    $$
    \ha  + \DNE{\Sigma_k} +\DNS{\U_k} \vdash \varphi \leftrightarrow \varphi';
    $$
     \item 
    \label{eq: item for Pi_k in PNFT}
    If $\varphi \in \U_k^+$, then there exists $\varphi' \in \Pi_k$ such that $\FV{\vp}=\FV{\vp'}$ and
    $$
    \ha + \DNE{(\Pi_k\lor \Pi_k)}  \vdash \varphi \leftrightarrow \varphi' ;
    $$
    \end{enumerate}
where $\DNS{\U_k}$ is the axiom scheme of the double-negation-shift restricted to formulas in $\U_k$ and  $\DNE{(\Pi_k\lor \Pi_k)}$ is ${\rm DNE}$  restricted to formulas of the form $\vp \lor \psi$ with $\vp, \psi \in \Pi_k$.
\end{theorem}

\begin{remark}
\label{rem: Sk-LEM -> Sk-DNE, Uk-DNS and PkvPk-DNE}
$\ha + \LEM{\Sigma_k}$ proves $\DNE{\Sigma_k}$, $\DNS{\U_k}$ and $\DNE{(\Pi_k \lor \Pi_k)}$.
Then the prenex normal form theorems for $\E_k^+$ and $\U_k^+$ are available in $\ha+\LEM{\Sigma_k}$.
\end{remark}

\begin{definition}[cf. {\cite[Definition 6]{Ishi00}}]
\label{def: Rk and Jk}
Define $\R_0:=\J_0:=\Sigma_0 \, (=\Pi_0)$.
In addition, we define simultaneously classes $\R_{k+1}$ and $\J_{k+1}$ as follows:
Let
$F$
range over formulas in
$\F_k^+$,
$R$ and $R'$ over those in $\R_{k+1}$, and $J$ and $J'$ over those in $\J_{k+1}$ respectively.
Then $\R_{k+1}$ and $\J_{k+1}$ are inductively generated by the clauses
\begin{enumerate}
    \item
    $F, R\land R', R \lor R', \forall x R, J \to R\in \R_{k+1}$;
        \item
    $F, J\land J', J \lor J', \exists x J, R \to J\in \J_{k+1}$.
\end{enumerate}
\end{definition}

\begin{lemma}[A relativized version of {\cite[Proposition 7.(2, 3)]{Ishi00}}]
\label{lem: A technical lemma for Rk and Jk in semi-callsical systems}
For a $\ha$-formula $\varphi$, the following hold$:$
\begin{enumerate}
    \item
    \label{item: property for R_{k+1}}
    If $\vp\in \R_{k+1}$, then $\ha^{\ph} + \LEM{\Sigma_{k}}$ proves $\neg_{\ph}\neg \vp \to \vp^{\ph};$
        \item
            \label{item: property for J_{k+1}}
    If $\vp\in \J_{k+1}$, then $\ha^{\ph} + \LEM{\Sigma_{k}}$ proves
    $\vp^{\ph} \to \neg_{\ph} \neg_{\ph} \vp$.
\end{enumerate}
\end{lemma}
\begin{proof}
We show items \ref{item: property for R_{k+1}} and \ref{item: property for J_{k+1}} simultaneously by induction on the structure of formulas.

Let $\vp $ be prime.
Since $\vp $ is in $\F_0$, we have $\vp \in \R_{k+1} \cap \J_{k+1}$.
Since $\ha \vdash \vp \lor \neg \vp$, we have $\ha^{\ph} \vdash \neg_{\ph} \neg \vp \to \vp\lor {\ph}$.
Then we have item \ref{item: property for R_{k+1}} by Lemma \ref{lem: HA|- Aqf* <-> Aqf v * }.
Item \ref{item: property for J_{k+1}} is trivial.

The induction step is the same as that for \cite[Proposition 7]{Ishi00} in addition with the cases of $\vp:\equiv \forall x \vp_1 \in \J_{k+1}$ and $\vp:\equiv \exists x \vp_1 \in \R_{k+1}$:

If $\vp:\equiv \forall x \vp_1 \in \J_{k+1}$, then we have $\vp \in \F_k^+$, and hence, $\vp \in \U_k^+$.
By Remark \ref{rem: Sk-LEM -> Sk-DNE, Uk-DNS and PkvPk-DNE},
one may assume $\vp\in \Pi_k$.
By Lemma \ref{Kurahashi Lemma: A* <-> A v *}.\eqref{item: A v * for Pi_{k+1}}, we have $\ha^{\ph} + \LEM{\Sigma_k} \vdash \vp^{\ph} \lr \vp \lor {\ph}$.
Since $\vp \lor {\ph}$ implies $ \neg_{\ph} \neg_{\ph} \vp$, we have $\ha^{\ph} + \LEM{\Sigma_k} \vdash \vp^{\ph} \to \neg_{\ph} \neg_{\ph} \vp$.

If $\vp:\equiv \exists x \vp_1 \in \R_{k+1}$, then we have $\vp \in \F_k^+$, and hence,
$\vp \in \E_k^+$ (and $k>0$).
By Remark \ref{rem: Sk-LEM -> Sk-DNE, Uk-DNS and PkvPk-DNE},
one may assume $ \vp_1\in \Pi_{k-1}$.
Reason in $\ha^{\ph}+\LEM{\Sigma_k}$.
Now we have $\exists x \vp_1 \lor \neg \exists x \vp_1$.
In the latter case, we have ${\ph}$ in the presence of $\neg_{\ph} \neg \exists x \vp_1$.
Thus we have $\neg_{\ph} \neg \exists x \vp_1 \to \exists x \vp_1 \lor \ph$.
By Corollary \ref{cor: A v * for Sigma_k}, we have that $\neg_{\ph} \neg \exists x \vp_1$ implies $\exists x \left( {\vp_1}^{\ph} \right)$, and hence, $\left(\exists x \vp_1 \right)^{\ph}$.
\end{proof}

\begin{definition}[cf. {\cite[Definition 6]{Ishi00}}]
\label{def: Qk}
Define  $\Q_0:=\Sigma_0\, (=\Pi_0)$.
In addition, we define a class $\Q_{k+1}$ as follows.
Let $P$ range over prime formulas,
$Q$ and $Q'$ over formulas in $\Q_{k+1}$, and $J$ over those in $\J_{k+1}$.
Then $\Q_{k+1}$ is inductively generated by the clause
  $$ P, Q\land Q', Q \lor Q', \forall x Q, \exists x Q,  J \to Q\in \Q_{k+1} .$$
\end{definition}

\begin{lemma}[A relativized version of {\cite[Proposition 7.(1)]{Ishi00}}]
\label{lem: For A in Qk+1, HA + SigmakLEM |- A -> A*}
For a $\ha$-formula $\varphi$, 
if $\vp\in \Q_{k+1}$, then $\ha^{\ph} + \LEM{\Sigma_k} \vdash \vp \to \vp^{\ph}$.
\end{lemma}
\begin{proof}
By induction on the structure of formulas, we show that for any $\ha$-formula $\vp$,  if $\vp\in \Q_{k+1}$, then $\ha^{\ph} + \LEM{\Sigma_k} \vdash \vp \to \vp^{\ph}$.

If $\vp$ is prime, then we have $\ha^{\ph} \vdash \vp \to \vp^{\ph}$ trivially by the definition of $\vp^{\ph}$.
If $\vp:\equiv \vp_1 \land \vp_2$, $\vp:\equiv \vp_1 \lor \vp_2$, $\vp:\equiv \forall x \vp_1$ or $\vp:\equiv \exists x \vp_1$, we have $\ha^{\ph} + \LEM{\Sigma_k} \vdash \vp \to \vp^{\ph}$ in a straightforward way by using the induction hypothesis (as for \cite[Proposition 7.(1)]{Ishi00}).

Assume $\vp :\equiv  \vp_1 \to \vp_2\in \Q_{k+1}$.
Then we have $\vp_1 \in \J_{k+1}$ and $\vp_2 \in \Q_{k+1}$.
By the induction hypothesis, we have $\ha^{\ph} + \LEM{\Sigma_k} \vdash \vp_2 \to \vp_2^{\ph}$.
On the other hand, by Lemma \ref{lem: A technical lemma for Rk and Jk in semi-callsical systems}.\eqref{item: property for J_{k+1}}, we have $\ha^{\ph} + \LEM{\Sigma_k} \vdash \vp_1^{\ph} \to \neg_{\ph} \neg_{\ph} \vp_1$.
Since $\ha^{\ph} \vdash \neg_{\ph}\neg_{\ph} \vp_2^{\ph} \lr  \vp_2^{\ph}$ by Proposition \ref{Soundness of *-translation}.\eqref{item: HA |- N*N*A* <-> A*}, we have that $\ha^{\ph} + \LEM{\Sigma_k}$ proves
$$
\begin{array}{rcl}
(\vp_1 \to \vp_2)& \underset{\text{[I.H.] }\LEM{\Sigma_k}}{\lra} & (\vp_1 \to \vp_2^{\ph})\\
&\longrightarrow& 
(\neg_{\ph} \neg_{\ph} \vp_1 \to  \neg_{\ph} \neg_{\ph} \vp_2^{\ph})\\[5pt]
&\underset{\LEM{\Sigma_k}}{\lra}& (\vp_1^{\ph} \to \neg_{\ph} \neg_{\ph} \vp_2^{\ph})\\
&\llr& (\vp_1^{\ph} \to \vp_2^{\ph}).
\end{array}
$$
\end{proof}

Now we define a class $\V_k$ of $\ha$-formulas by using the class $\J_k$ in Definitions \ref{def: Rk and Jk}.
\begin{definition}
\label{def: Vk}
Let $J$ range over formulas in $\J_k$, $V$ and $V'$ over those in $\V_{k}$.
Then $\V_{k}$ is inductively generated by the clause
  $$ J, V\land V', \forall x V \in \V_{k} .$$
\end{definition}

For our conservation result, we use the following fact on substitution.
\begin{lemma}[cf. {\cite[Theorem 6.2.4]{vD13} and \cite[Lemma 6.10]{FK20-1}}]
\label{lem: Substitution}
Let $X$ be a set of $\ha$-sentences and $\vp$ be a $\ha^{\ph}$-formula.
If $\ha^{\ph} + X \vdash \vp$, then $\ha + X \vdash \vp[\psi/{\ph}]$ for any $\ha$-formula $\psi$ such that the free variables of $\psi$ are not bounded in $\vp$, where $\vp[\psi/{\ph}]$ is the $\ha$-formula obtained from $\vp$ by replacing all the occurrences of ${\ph}$ in $\vp$ with $\psi$.
\end{lemma}

\begin{theorem}
\label{thm: conservativity on semi-classical arithmetic}
For any $\ha$-formulas $\vp\in \V_{k+1}$ and $\psi \in \Q_{k+1}$, if $ \pa \vdash \psi \to \vp $, then $\ha +\LEM{\Sigma_k} \vdash \psi \to \vp$.
\end{theorem}
\begin{proof}
Since one can freely replace the bounded variables, it suffices to show that for any $\ha$-formulas $\vp\in \V_{k+1}$ and $\psi \in \Q_{k+1}$ such that the free variables of $\vp$ are not bounded in $\psi$, if $ \pa \vdash \psi \to \vp $, then $\ha +\LEM{\Sigma_k} \vdash \psi \to \vp$.
We show this assertion by induction on the structure of formulas in $\V_{k+1}$.

Case of $\vp \in \J_{k+1}$:
Fix $\psi \in \Q_{k+1}$ such that the free variables of $\vp$ are not bounded in $\psi$.
Suppose $\pa \vdash \psi \to \vp$.
Then, by Proposition \ref{Soundness of *-translation}.\eqref{item: PA |- A => HA |- A*}, we have $\ha^{\ph} \vdash \psi^{\ph} \to \vp^{\ph} $.
By Lemma \ref{lem: For A in Qk+1, HA + SigmakLEM |- A -> A*} and Lemma \ref{lem: A technical lemma for Rk and Jk in semi-callsical systems}.\eqref{item: property for J_{k+1}}, we have $\ha^{\ph} +\LEM{\Sigma_k} \vdash \psi \to \neg_{\ph} \neg_{\ph} \vp$.
By Lemma \ref{lem: Substitution}, we have that  $\ha +\LEM{\Sigma_k}$ proves $\psi \to ((\vp \to \vp) \to \vp)$, equivalently, $\psi \to \vp$.

Case of $\vp:\equiv \vp_1 \land \vp_2 \in \V_{k+1}$:
Then $\vp_1 , \vp_2 \in \V_{k+1}$.
Fix $\psi \in \Q_{k+1}$ such that the free variables of $\vp_1 \land \vp_2$ are not bounded in $\psi$.
Suppose $\pa \vdash \psi \to \vp_1 \land \vp_2$.
Then $\pa \vdash \psi \to \vp_1 $ and $\pa \vdash \psi \to \vp_2$.
By the induction hypothesis, we have $\ha +\LEM{\Sigma_k} \vdash  \psi \to \vp_1 $ and  $\ha +\LEM{\Sigma_k} \vdash  \psi \to \vp_2 $, and hence,  $\ha +\LEM{\Sigma_k} \vdash \psi \to \vp_1 \land \vp_2 $.

Case of $\vp: \equiv \forall x \vp_1  \in \V_{k+1}$:
Then $ \vp_1 \in \V_{k+1}$.
Fix $\psi \in \Q_{k+1}$ such that the free variables of $\forall x \vp_1$ are not bounded in $\psi$.
In addition, assume that $x$ does not appear in $\psi$ without loss of generality.
Suppose $\pa \vdash \psi \to \forall x \vp_1$.
Then $\pa \vdash \psi \to \vp_1$.
By the induction hypothesis, we have that $\ha +\LEM{\Sigma_k}$ proves  $\psi \to   \vp_1$.
Since $x \notin \FV{\psi}$, we have $\ha +\LEM{\Sigma_k} \vdash  \psi \to \forall x \vp_1$.
\end{proof}

\begin{remark}
\label{rem: our conservation result is an extension of the previous one}
Since $\Pi_{k+2} $ is a sub-class of $\V_{k+1}$ and $\Q_{k+1}$ contains all prenex formulas, we have \cite[Theorem 6.14]{FK20-1} (and a-fortiori Proposition \ref{prop: PA is Pi_k+2-cons. over HA + Sigma_k-LEM}) as a corollary of Theorem \ref{thm: conservativity on semi-classical arithmetic}. 
\end{remark}

\begin{corollary}
\label{cor: Conservation result with X}
Let
$X$ be a set of $\ha$-sentences in $\Q_{k+1}$.
For any $\ha$-formulas $\vp\in \V_{k+1}$ and $\psi \in \Q_{k+1}$, if $ \pa + X \vdash \psi \to \vp $, then $\ha +X +\LEM{\Sigma_k} \vdash \psi \to \vp$.
\end{corollary}
\begin{proof}
Assume $\pa +X \vdash \psi \to \vp $.
Then there exists a finite number of sentences $\psi_0, \dots, \psi_m \in X $ such that $\pa + \psi_0+ \dots + \psi_m  \vdash  \psi \to \vp$.
Since $\pa$ satisfies the deduction theorem, we have $\pa \vdash \psi_0 \land \dots \land \psi_m \land \psi \to \vp$.
Since $\psi_0 \land \dots \land \psi_m \land \psi \in \Q_{k+1}$, by  Theorem \ref{thm: conservativity on semi-classical arithmetic}, we have
 $\ha +\LEM{\Sigma_k} \vdash \psi_0 \land \dots \land \psi_m \land \psi \to \vp$, and hence,
 $\ha + X +\LEM{\Sigma_k} \vdash \psi \to \vp$.
\end{proof}

\section{The relation of the classes $\R_k$ and $\J_k$ with the existing classes $\U_k$ and $\E_k$}
\label{sec: formula-classes}
In the following, we show that our classes $\R_k$ and $\J_k$ in Definition \ref{def: Rk and Jk} are in fact equivalent over $\ha$ to $\U_k $ and $\E_k$ (see Definition \ref{def: Classes}) respectively.

\begin{proposition}
\label{prop: Ukp=Rk, Ekp=Jk}
$\U_k^+=\R_k$ and $\E_k^+ = \J_k$.
\end{proposition}
\begin{proof}
By induction on $k$.
The base case is trivial.
For the induction step, assume $\U_k^+=\R_k$ and $\E_k^+ = \J_k$.
We show 
\begin{enumerate}
    \item 
    \label{item: Rk+1=Uk+1p}
    $\vp \in  \U_{k+1}^+$ if and only if $\vp \in \R_{k+1}$
    \item 
    \label{item: Jk+1=Ek+1p}
    $\vp \in  \E_{k+1}^+$ if and only if $\vp \in \J_{k+1}$
\end{enumerate}
simultaneously by induction on the structure of formulas.
If $\vp$ is prime, since $\vp \in \F_0$, we are done.
Assume that items \ref{item: Rk+1=Uk+1p} and \ref{item: Jk+1=Ek+1p} hold for $\vp_1$ and $\vp_2$.
Using \cite[Lemma 4.5.(1)]{FK20-1},
we have
$$\vp_1 \land \vp_2 \in \U_{k+1}^+ \LR \vp_1, \vp_2 \in \U_{k+1}^+ \underset{\text{I.H.}}{\LLR}  \vp_1, \vp_2 \in \R_{k+1} \LR \vp_1 \land \vp_2 \in \R_{k+1} .$$
In the same manner, we also have
$\vp_1 \land \vp_2 \in \E_{k+1}^+ \LR \vp_1 \land \vp_2\in \J_{k+1}$, $\vp_1 \lor \vp_2 \in \U_{k+1}^+ \LR \vp_1 \lor \vp_2\in \R_{k+1}$,
$\vp_1 \lor \vp_2 \in \E_{k+1}^+ \LR \vp_1 \lor \vp_2\in \J_{k+1}$.
For $\vp_1 \to \vp_2$,
using \cite[Lemma 4.5.(3)]{FK20-1}
we have
$$
\begin{array}{cl}
     & \vp_1 \to \vp_2 \in \U_{k+1}^+  \\[2pt]
     \LLR & \vp_1 \in \E_{k+1}^+ \text{ and } \vp_2 \in \U_{k+1}^+\\[2pt]
   \underset{\text{I.H.}}{\LLR} & \vp_1 \in J_{k+1} \text{ and } \vp_2 \in \R_{k+1}  \\
   \LLR & \vp_1 \to \vp_2 \in \R_{k+1}.
\end{array}
$$
In the same manner, we also have
$\vp_1 \to \vp_2 \in \E_{k+1}^+ \LR \vp_1 \to \vp_2\in \J_{k+1}$.
For $\forall x \vp_1$,
using \cite[Lemma 4.5.(4,6)]{FK20-1},
we have
$$
\forall x \vp_1 \in \U_{k+1}^+ \LR \vp_1\in \U_{k+1}^+ \underset{\text{I.H.}}{\LLR} \vp_1\in \R_{k+1} \LR \forall x \vp_1\in \R_{k+1},
$$
and
$$
\forall x \vp_1 \in \E_{k+1}^+ \LR \forall x \vp_1\in \U_{k}^+ \LR \forall x \vp_1\in  \F_k^+ \LR \forall x \vp_1\in \J_{k+1}.
$$
In the same manner, we also have $\exists x \vp_1 \in \U_{k+1}^+ \LR \exists x \vp_1 \in \R_{k+1}$ and $\exists x \vp_1 \in \E_{k+1}^+ \LR \exists x \vp_1 \in \J_{k+1}$.
\end{proof}

\begin{corollary}
\label{cor: Uk=Rk, Ek=Jk}
$\U_k=\R_k$ and $\E_k = \J_k$.
\end{corollary}
\begin{proof}
Immediate by Proposition \ref{prop: Ukp=Rk, Ekp=Jk} and Remark \ref{rem: On the non-commutativity of Ek, Uk and Fk}.
\end{proof}

\begin{corollary}
\label{cor: PA is Ek+1-cons. over HA+SkLEM}
For a set $X$ of $\ha$-sentences in $\Q_{k+1}$,
$\pa +X$ is $\E_{k+1}$-conservative over $\ha +X + \LEM{\Sigma_k}$.
\end{corollary}
\begin{proof}
Immediate from
Corollaries \ref{cor: Conservation result with X} and
\ref{cor: Uk=Rk, Ek=Jk}
since $\J_{k+1}\subseteq \V_{k+1}$.
\end{proof}

\begin{remark}
Corollary \ref{cor: PA is Ek+1-cons. over HA+SkLEM} deals with the conservativity of the class of formulas in $\E_{k+1}$, which
seems to be strictly stronger than that for sentences in $\E_{k+1}$ (cf. Section \ref{sec: Cons. theorems for utS_k and utEk}).
\end{remark}

\begin{remark}
\label{rem: Rk' and Jk'}
Similar to Definition \ref{def: Rk and Jk}, define the classes $\R'_k$ and $\J'_k$ as follows.
Define  $\R'_0:=\J'_0:=\Sigma_0 \, (=\Pi_0)$ and $\R'_{k+1}$ and $\J'_{k+1}$  simultaneously as follows:
Let
$E$
range over formulas in $\E_k^+$,
$U$
over those in $\U_k^+$,
$R$ and $R'$ over those in $\R'_{k+1}$, and $J$ and $J'$ over those in $\J'_{k+1}$ respectively.
Then $\R'_{k+1}$ and $\J'_{k+1}$ are inductively generated by the clauses
\begin{enumerate}
    \item
    $E, R\land R', R \lor R', \forall x R, J \to R\in \R'_{k+1}$;
        \item
    $U, J\land J', J \lor J', \exists x J, R \to J\in \J'_{k+1}$.
\end{enumerate}
Then the proof of Proposition \ref{prop: Ukp=Rk, Ekp=Jk} shows that $\U_k^+=\R'_k$ and $\E_k^+ = \J'_k$.
Hence $\R_k=\R'_k$ and $\J_k = \J'_k$.
\end{remark}

\begin{remark}
\label{rem: Rk'' and Jk''}
Define $\R_{k+1}''$ and $\J_{k+1}''$ as for $\R_{k+1}'$ and $\J_{k+1}'$ in Remark \ref{rem: Rk' and Jk'} with replacing $\E_k^+$ and $\U_k^+$ by $\Sigma_k$ and $\Pi_k$.
Then, as in the proof of Proposition \ref{prop: Ukp=Rk, Ekp=Jk} with using the prenex normal form theorems in $\ha +\LEM{\Sigma_k}$ (cf. Remark \ref{rem: Sk-LEM -> Sk-DNE, Uk-DNS and PkvPk-DNE}), one can show $\U_{k+1}^+=\R''_{k+1}$ and $\E_{k+1}^+ = \J''_{k+1}$ over $\ha+\LEM{\Sigma_k}$.
\end{remark}

As described in Definition \ref{def: Classes}, the classes $\E_k$ and  $\U_k$ are originally defined by using the notion of alternation path.
On the other hand,
Remark \ref{rem: Rk'' and Jk''} reveals that one can define these classes (via Remark \ref{rem: On the non-commutativity of Ek, Uk and Fk}) inductively without using the notion of alternation path.
A technical advantage of this usual way of defining classes is that one can prove properties of these classes by induction on the structure of formulas in those classes.

\section{Conservation theorems for the classes of formulas}
\label{sec: CONS for classes of formulas}
In this section, we explore the notion that $\pa $ is $\Gamma$-conservative over $\T$ for semi-classical arithmetic $\T$ and a class $\Gamma$ of formulas (especially, $\Pi_k, \Sigma_k, \U_k, \E_k, \F_k$ etc.).
\begin{definition}
For classes of $\ha$-formulas $\Gamma$ and $\Gamma'$,
$\Gamma \lor \Gamma'$ is the class of formulas of form $\vp \lor \psi$ where $\vp\in \Gamma $ and $\psi \in \Gamma'$.
\end{definition}

We recall the notion of duals for prenex formulas from \cite{ABHK04, FK20-2}.
\begin{definition}[cf.~{\cite[Definition 3.2]{FK20-2}}]
\label{def: duals of prenex formulas}
For any formula $\varphi$ in prenex normal form, we define the dual $\varphi^\perp$ of $\varphi$ inductively as follows: 
\begin{enumerate}
	\item $\varphi^\perp : \equiv \neg \varphi$ if $\varphi$ is quantifier-free; 
	\item $(\forall x \varphi)^\perp : \equiv \exists x (\varphi)^\perp$; 
	\item $(\exists x \varphi)^\perp : \equiv \forall x (\varphi)^\perp$. 
\end{enumerate}
\end{definition}
\begin{remark}
\label{rem: duals}
For $\vp$ in $\Sigma_k$ (resp. $\Pi_k$), $\vp^\perp$ is in $\Pi_k$ (resp. $\Sigma_k$), $\FV{\vp^{\perp}}=\FV{\vp}$ and $\left( \vp^\perp \right)^\perp$ is equivalent to $\vp$ over $\ha$.
For each prenex formula $\vp$, $\vp^\perp$ implies $\neg \vp$ intuitionistically.
On the other hand, the converse direction for formulas in $\Sigma_k$ (resp. $\Pi_k$) is equivalent to $\DNE{\Sigma_{k-1}}$ (resp. $\DNE{\Sigma_{k}}$).
Then it follows that for $\vp \in \Sigma_k$ there exists $\vp'\in \Pi_k$ such that $\FV{\vp'}=\FV{\vp}$ and $\ha +\DNE{\Sigma_{k-1}} \vdash \vp' \lr \neg \vp$ (cf. \cite[Lemma 4.8.(2)]{FK20-1}).
In addition,  $\neg \vp^\perp$ implies $\neg \neg  \vp$ in the presence of $\DNE{\Sigma_{k-1}}$ for the both cases of $\vp\in \Sigma_k$ and $\vp\in \Pi_k$.
Note also that $\pa$ proves $\vp \lor \vp^\perp$ for each prenex formula $\vp$.
We refer the reader to \cite[Section 3]{FK20-2} for more information about the dual principles for prenex formulas in semi-classical arithmetic.
\end{remark}

\subsection{Conservation theorems for $\Pi_{k}, \Sigma_k, \E_k$ and $\F_k$}
\label{subsec: Ek+1 and Fk cons.}
\begin{definition}
Let $\T$ be a theory in the language of $\ha$ and $\Gamma$ be a class of $\ha$-formulas.
\begin{itemize}
    \item
    $\T$ is closed under $\DNER{\Gamma}$ if $\T \vdash \neg\neg \vp $ implies $\T \vdash \vp$ for all $\vp\in \Gamma$.
    \item
    $\T$ is closed under $\CDR{\Gamma}$ if $\T \vdash \forall x (\vp \lor \psi) $ implies $\T \vdash\vp \lor \forall x \psi$ for all  $\vp,\psi \in \Gamma$ such that $x \notin \FV{\vp}$.
    \item
    $\T$ is closed under $\DMLR{\Gamma}$ (resp. $\DMLDR{\Gamma}$) if $\T \vdash \neg(\vp \land \psi) $ implies $\T \vdash \neg \vp \lor \neg \psi$ (resp. $\T \vdash \vp^\perp \lor \psi^\perp$) for all $\vp, \psi\in \Gamma$.
\end{itemize}
Note that $\vp$ and $\psi$ in the above may contain free variables.
\end{definition}

As mentioned in \cite[Section 3.5.1]{ConstMathI}, $\DNER{\Sigma_1}$ is known as Markov's rule (for primitive recursive predicates).
The fact that $\pa$ is $\Sigma_1$-conservative (equivalently, $\Pi_2$-conservative) over $\ha$ implies that $\ha$ is closed under Markov's rule ($\DNER{\Sigma_1}$), and vice versa.
The generalization $\DNER{\Sigma_k}$ of Markov's rule is already mentioned in \cite[Section 4.4]{HN02}.
It is easy to see that for semi-classical arithmetic $\T$, if $\pa$ is $\Sigma_k$-conservative over $\T$, then $\T$ is closed under $\DNER{\Sigma_k}$.
Then it is natural to ask about the converse.
As we show in Theorem \ref{thm: conservation theorems equivalent to SkLEM} below, this is also the case (note that the case for $k=2$ is essentially shown in the proof of \cite[Proposition 3.3]{KS14}).

The following are our ``reversal'' results.
\begin{lemma}
\label{lem: SkvPik-cons->Sk-LEM}
Let $\T $ be a
theory containing $\ha$.
If $\pa$ is $\left(\Sigma_k\lor \Pi_k\right)$-conservative over $\T$,
then $T \vdash \LEM{\Sigma_{k}}$.
\end{lemma}
\begin{proof}
Fix $\xi \in \Sigma_{k}$.
Let $\xi^\perp\in \Pi_{k}$ be the dual of $\xi$.
Since $\pa \vdash \xi \lor \xi^\perp$, by our assumption, we have $\T \vdash \xi \lor \xi^\perp$, and hence, $\T \vdash \xi \lor \neg \xi$.
\end{proof}

\begin{lemma}
\label{lem: Sk+1-DNER => Sk-LEM}
Let $\T$ be a theory containing $\ha$.
If $\T$ is closed under $\DNER{\Sigma_{k+1}}$, then $T$ proves $\LEM{\Sigma_k}$.
\end{lemma}
\begin{proof}
We show that for all $m\leq k$, $\T$ proves $\LEM{\Sigma_m}$, by induction on $m$.
Since $\T$ contains $\ha$, the base case is trivial.
Assume $m+1\leq k$ and $\T \vdash \LEM{\Sigma_m}$.
Let $\vp\in \Sigma_{m+1}$.
Since $\ha \vdash \neg \neg (\vp \lor \neg \vp)$, by Remark \ref{rem: duals} and the fact that $\LEM{\Sigma_m}$ implies $\DNE{\Sigma_m}$, we have $\T \vdash \neg \neg (\vp \lor \vp^\perp)$ where $\vp^\perp \in \Pi_{m+1}$.
Since $\vp \lor \vp^\perp$ is equivalent over $\ha$ to some formula in $\Sigma_{m+2}$ (cf. \cite[Lemma 4.4]{FK20-1}), by $\DNER{\Sigma_{k+1}}$, we have
$\T \vdash \vp \lor \vp^\perp$,
and hence, $\vp \lor \neg \vp$.
Thus we have shown $\T \vdash \LEM{\Sigma_{m+1}}$.
\end{proof}

\begin{lemma}\label{lem: Pk+1-CDR => Sk-LEM}
Let 
$\T$ be a theory containing $\ha$.
If $\T$ is closed under $\CDR{\Sigma_{k}}$, then $T$ proves $\LEM{\Sigma_k}$.
\end{lemma}
\begin{proof}
We show that for all $m\leq k$, $\T$ proves $\LEM{\Sigma_m}$, by induction on $m$.
Since $\T$ contains $\ha$, the base case is trivial.
Assume $m+1\leq k$ and $\T \vdash \LEM{\Sigma_m}$.
Let $\vp :\equiv \exists x \vp_1$ where $\vp_1\in \Pi_m$.
Since $\T$ proves $\LEM{\Pi_m}$ and $\DNE{\Sigma_m}$, we have $\T \vdash \vp_1 \lor \neg \vp_1$, and hence, $\T\vdash \vp_1 \lor \vp_1^\perp$ (cf. Remark \ref{rem: duals}).
Then $\T \vdash \forall x(\exists x \vp_1 \lor \vp_1^\perp)$ follows.
Since $\exists x \vp_1 ,\vp_1^\perp \in \Sigma_{m+1}$,
by $\CDR{\Sigma_{k}}$,
we have $\T \vdash \exists x \vp_1 \lor \forall x \vp_1^\perp$, and hence, $\T \vdash \exists x \vp_1 \lor \neg \exists x \vp_1$.
Thus we have shown $\T \vdash \LEM{\Sigma_{m+1}}$.
\end{proof}

\begin{lemma}
\label{lem: Pk-DMLDR <=> Sk-DNER}
Let  $\T$ be a theory containing $\ha$.
Then  $\T$ is closed under $\DMLDR{\Pi_{k}}$ if and only if $\T$ is closed under $\DNER{\Sigma_{k}}$.
\end{lemma}
\begin{proof}
We first show the ``only if'' direction.
Assume that $\T$ is closed under $\DMLDR{\Pi_{k}}$ and  $\T  \vdash \neg \neg \vp$ where $\vp\in \Sigma_k$.
Since $\neg\neg \vp$ is equivalent over $\ha$ to $\neg (\neg \vp \land \neg \vp)$, by Remark \ref{rem: duals}, we have
$$
\T \vdash \neg (\vp^\perp \land \vp^\perp).
$$
Since $\vp^\perp \in \Pi_k$, by $\DMLDR{\Pi_{k}}$, we have
$\T \vdash \left(\vp^\perp\right)^\perp \lor \left( \vp^\perp \right)^\perp$, and hence, $\T \vdash  \vp$ (cf. Remark \ref{rem: duals}). 

For the converse direction, assume that $\T$ is closed under 
 $\DNER{\Sigma_{k}}$ and $\T \vdash \neg (\vp \land \psi)$ where $\vp, \psi\in \Pi_k$.
Since $\neg(\vp \land \psi)$ is intuitionistically equivalent to $\neg(\neg\neg \vp \land \neg\neg \psi)$, by Lemma \ref{lem: Sk+1-DNER => Sk-LEM} and Remark \ref{rem: duals} (note that $\LEM{\Sigma_{k-1}}$ implies $\DNE{\Sigma_{k-1}}$), we have $\T \vdash \neg \left(\neg \vp^\perp \land \neg \psi^\perp \right)$ where $\vp^\perp , \psi^\perp \in \Sigma_k$.
Then $\T \vdash \neg \neg \left(\vp^\perp \lor \psi^\perp \right)$ follows.
By $\DNER{\Sigma_{k}}$, we have $\T \vdash \vp^\perp \lor \psi^\perp$.
\end{proof}

\begin{theorem}
\label{thm: conservation theorems equivalent to SkLEM}
Let
$\T$ be semi-classical arithmetic and $X$ be a set of $\ha$-sentences in $\Q_{k+1}$.
The following are pairwise equivalent$:$
\begin{enumerate}
    \item 
    \label{item: CONS(Vk+1)}
    $\pa +X$ is $\V_{k+1}$-conservative over $\T +X ;$
    \item
        \label{item: CONS(Pik+2)}
        $\pa +X$ is $\Pi_{k+2}$-conservative over $\T+X ;$
        \item
  \label{item: CONS(Sk+1)}
        $\pa +X$ is $\Sigma_{k+1}$-conservative over $\T+X ;$
\item
\label{item: Sk+1DNER}
$\T+X$ is closed under $\DNER{\Sigma_{k+1}};$
\item
\label{item: Pk+1DMLDR}
$\T+X$ is closed under $\DMLDR{\Pi_{k+1}};$
\item
    \label{item: CONS(Ek+1p)}
    $\pa +X$ is $\E_{k+1}$-conservative over $\T+X ;$
\item
    \label{item: CONS(Fkp)}
    $\pa +X$ is $\F_{k}$-conservative over $\T+X ;$
    \item
    \label{item: CONS(SkvPk)}
    $\pa +X$ is $\left(\Sigma_k\lor \Pi_k\right)$-conservative over $\T+X ;$
    \item
        \label{item: T|-SkLEM}
    $\T +X \vdash \LEM{\Sigma_k};$
        \item
        \label{item: T|-SkCD}
    $\T +X \vdash \CD{\Sigma_k};$
    \item
        \label{item: T|-SkCDR}
    $\T +X $ is closed under $\CDR{\Sigma_k};$
\end{enumerate}
where $\CD{\Sigma_k}$ is the scheme
$\forall x (\vp \lor \psi) \to \vp \lor \forall x \psi$
with $\vp,\psi \in \Sigma_k$ such that $x \notin \FV{\vp}$ (cf. \cite[Section 7]{FK20-2}).
\end{theorem}
\begin{proof}
The implications $\eqref{item: CONS(Vk+1)} \to \eqref{item: CONS(Ek+1p)} \to \eqref{item: CONS(Fkp)}\to \eqref{item: CONS(SkvPk)}$, $\eqref{item: CONS(Vk+1)} \to \eqref{item: CONS(Pik+2)} \to \eqref{item: CONS(Sk+1)}\to \eqref{item: Sk+1DNER}$ and
$\eqref{item: T|-SkLEM} \to \eqref{item: T|-SkCD} \to \eqref{item: T|-SkCDR}$
are trivial (cf. Corollary \ref{cor: PA is Ek+1-cons. over HA+SkLEM} and Remark \ref{rem: our conservation result is an extension of the previous one}).
The implications $\eqref{item: CONS(SkvPk)}\to \eqref{item: T|-SkLEM}$,
$\eqref{item: Sk+1DNER} \to \eqref{item: T|-SkLEM}$, $\eqref{item: T|-SkCDR} \to \eqref{item: T|-SkLEM}$ and $\eqref{item: T|-SkLEM} \to \eqref{item: CONS(Vk+1)}$ are by Lemmata \ref{lem: SkvPik-cons->Sk-LEM}, \ref{lem: Sk+1-DNER => Sk-LEM}, \ref{lem: Pk+1-CDR => Sk-LEM} and
Corollary \ref{cor: Conservation result with X} respectively.
The equivalence $\eqref{item: Sk+1DNER} \lr \eqref{item: Pk+1DMLDR}$ is by Lemma \ref{lem: Pk-DMLDR <=> Sk-DNER}.
\end{proof}

\subsection{Conservation theorem for $\U_k$}
\label{subsec: Uk+1 cons.}
In contrast to the fact that $\E_{k+1}$-conservativity and $\F_{k}$-conservativity are characterized by $\LEM{\Sigma_k}$ (see Theorem \ref{thm: conservation theorems equivalent to SkLEM}), 
$\U_{k+1}$-conservativity requires more than $\LEM{\Sigma_k}$:

\begin{proposition}
\label{prop: PA is not Pi1vPi1-conservative over HA}
$\pa $ is not $(\Pi_1\lor \Pi_1)$-conservative over $\ha$.
\end{proposition}
\begin{proof}
We use the same argument as in \cite[Section 3]{FK20-1}.
Suppose that $\pa $ is conservative over $\ha$ for all formulas $\vp \lor \psi$ with $\vp, \psi \in \Pi_1$.
Let $\Phi(x)$ be the following formula:
\begin{equation}
\label{eq: Psi(x)}
\forall u \neg ({\rm T}(x,x,u) \land {\rm U}(u)=0) \lor  \forall u \neg ({\rm T}(x,x,u) \land {\rm U}(u)\neq 0),
\end{equation}
where ${\rm T}$ and ${\rm U}$ are the standard primitive recursive predicate and the function from the Kleene normal form theorem.
Since
$$\neg \left( \exists u ({\rm T}(x,x,u) \land {\rm U}(u)=0) \land \exists u ({\rm T}(x,x,u) \land {\rm U}(u)\neq 0) \right)$$
is provable in $\ha$,
we have $\pa \vdash \Phi(x)$.
Then, by our assumption, we have $\ha \vdash \Phi(x)$, and hence, $\ha \vdash \forall x \Phi(x)$.
On the other hand, as shown in the proof of \cite[Proposition 3.1]{FK20-1}, $\neg \forall x \Phi(x) $ is provable in $\ha + {\rm CT_0}$ where ${\rm CT_0}$ is  the arithmetical form of Church's thesis from \cite[Section 3.2.14]{Tro73}.
Then we have $\ha +  {\rm CT_0} \vdash \perp$, which is a contradiction by \cite[Section 3.2.22]{Tro73}.
\end{proof}

Let $\T$ be semi-classical arithmetic.
By Theorems \ref{thm: PNFT}.\eqref{eq: item for Pi_k in PNFT} and \ref{thm: conservation theorems equivalent to SkLEM}, if $\T$ proves $\DNE{(\Pi_{k+1} \lor \Pi_{k+1})}$, then $\pa$ is $\U_{k+1}$-conservative (and hence, a-fortiori $(\Pi_{k+1} \lor \Pi_{k+1})$-conservative) over $\T$.
On the other hand, if $\pa$ is $(\Pi_{k+1} \lor \Pi_{k+1})$-conservative over $\T$, then $\T$ proves $\LEM{\Sigma_k}$ by Lemma \ref{lem: SkvPik-cons->Sk-LEM} and the fact that both of $\Sigma_k$ and $\Pi_k$ can be seen as sub-classes of $\Pi_{k+1}$.
Thus $\DNE{(\Pi_{k+1} \lor \Pi_{k+1})}$ implies the $\U_{k+1}$-conservativity, which implies the $(\Pi_{k+1} \lor \Pi_{k+1})$-conservativity, which implies $\LEM{\Sigma_k}$ and not vice versa.
For further studying the relation of the $\U_{k+1}$/$(\Pi_{k+1} \lor \Pi_{k+1})$-conservativity and semi-classical arithmetic, we introduce some extended classes of $\Pi_k$ and $\Sigma_k$.
\begin{definition}
\noindent
\begin{itemize}
\item
$\bigvee \Pi_k$ denotes the class consisting of disjunctions of formulas in $\Pi_k$. 
\item
A class $\EP_k$ is defined by the following clauses:
\begin{itemize}
    \item 
    $\vp\in \Pi_k$;
\item
If $\vp, \psi\in \EP_k$, then $\vp \lor \psi \in \EP_k$;
\item
If $\vp\in \EP_k$, then $\forall x \vp\in \EP_k$.
\end{itemize}
\item
$\ES_{k+1}$ denotes the class consisting of formulas of the form $\exists x_1, \dots, x_n \vp $ where $\vp\in \EP_k$.
\end{itemize}
\end{definition}

\begin{remark}\label{rem: class hierarchy between Pi_k and ESk+1}
$\Pi_k \subseteq \Pi_k \lor \Pi_k \subseteq \bigvee \Pi_k  \subseteq \EP_k \subseteq \ES_{k+1}$.
\end{remark}

\begin{lemma}
\label{lem: conjunction of EPk formulas}
For any $\ha$-formulas $\vp, \psi\in \EP_k$, there exists $\xi \in \EP_k$ such that $\FV{\xi}=\FV{\vp \land \psi}$ and $\ha \vdash \xi \lr \vp \land \psi$.
\end{lemma}
\begin{proof}
By induction on the sum of the complexity of $\vp $ and $ \psi$.

If both of $\vp$ and $ \psi$ are in $\Pi_k$, then we are done by \cite[Lemma 4.3.(2)]{FK20-1}.

Suppose $\psi:\equiv \psi_1\lor \psi_2$ where $\psi_1, \psi_2 \in \EP_k$.
By the induction hypothesis, there exist $\xi_1, \xi_2\in \EP_k$ such that
$\FV{\xi_1}=\FV{\vp \land \psi_1}$, $\FV{\xi_2}=\FV{\vp \land \psi_2}$, $\ha \vdash \xi_1 \lr \vp \land \psi_1$ and $\ha \vdash \xi_2 \lr \vp \land \psi_2$.
Then we have that $$\FV{ \xi_1 \lor \xi_2} = \FV{ \xi_1} \cup \FV{ \xi_2}=\FV{ \vp \land \psi_1} \cup \FV{ \vp \land \psi_2}= \FV{ \vp \land \psi}$$
and that $\ha$ proves
$$\xi_1 \lor \xi_2 \lr (\vp \land \psi_1)\lor (\vp \land \psi_2) \lr \vp \land (\psi_1 \lor \psi_2) \equiv \vp \land \psi.$$
Thus one can take $\xi_1 \lor \xi_2 \in \EP_k$ as a witness.

Suppose $\psi:\equiv \forall x \psi_1$ where $\psi_1 \in \EP_k$.
Without loss of generality, assume $x\notin \FV{\vp}$.
By the  induction hypothesis, there exists $\xi_1\in \EP_k$ such that $\FV{\xi_1}=\FV{\vp \land \psi_1}$ and $\ha \vdash \xi_1 \lr \vp \land \psi_1$.
Then we have
$$\FV{\forall x \xi_1} = \FV{\vp\land \psi_1}\setminus \{x\} = \FV{ \vp\land \forall x \psi_1}$$
and that $\ha$ proves
$$
\forall x \xi_1 \lr \forall x (\vp \land \psi_1) \lr \vp \land \forall x \psi_1.
$$
Thus one can take $\forall x \xi_1  \in \EP_k$ as a witness.
\end{proof}

In what follows, we use \cite[Lemma 4.5]{FK20-1} many times implicitly.
\begin{lemma}
\label{lem: U_k and EPk}
For a $\ha$-formula $\vp$, the following hold$:$
\begin{enumerate}
    \item 
    If $\vp \in \U_{k+1}^+$, then there exists $\vp'\in \EP_{k+1}$ such that $\FV{\vp} = \FV{\vp'}$, $\ha+\LEM{\Sigma_k} \vdash \vp' \to \vp$ and $\pa \vdash \vp \to\vp' ;$
    \item
    If $\vp \in \E_{k+1}^+$, then there exists $\vp'\in \EP_{k+1}$ such that $\FV{\vp} = \FV{\vp'}$, $\ha+\LEM{\Sigma_k} \vdash \vp' \to \neg \vp$ and $\pa \vdash \neg \vp \to\vp'$.
\end{enumerate}
\end{lemma}
\begin{proof}
We show items $1$ and $2$ by simultaneous induction on the structure of formulas.
We suppress the arguments on the condition on free variables when they are clear from the context.

If $\vp$ is prime, then items $1$ and $2$ are trivial since $\vp$ is decidable in $\ha$.
For the induction step, assume items $1$ and $2$ hold for $\vp_1$ and $\vp_2$.

Case of $\vp:\equiv \vp_1 \lor \vp_2$:
For item $1$, suppose $\vp_1 \lor \vp_2 \in \U_{k+1}^+$.
Then $\vp_1 , \vp_2 \in \U_{k+1}^+$.
By using the induction hypothesis, there exist $\vp_1', \vp_2'\in \EP_{k+1}$ such that
$\ha+\LEM{\Sigma_k} $ proves $\vp_1' \to \vp_1$ and $\vp_2' \to  \vp_2$ and $\pa $ proves $\vp_1 \to\vp_1'$ and $\vp_2 \to\vp_2'$.
Now $\vp_1' \lor \vp_2' \in \EP_{k+1}$ and $\ha +\LEM{\Sigma_k}$ proves
$$
\vp_1' \lor \vp_2'  \underset{\text{[I.H.] }\LEM{\Sigma_k}}{\lra} \vp_1 \lor \vp_2 .
$$
On the other hand, $\pa$ proves the converse.
For item $2$, suppose $\vp_1 \lor \vp_2\in \E_{k+1}^+$.
Then $\vp_1 , \vp_2 \in \E_{k+1}^+$.
By the induction hypothesis, there exist $\vp_1', \vp_2'\in \EP_{k+1}$ such that
$\ha+\LEM{\Sigma_k} $ proves $\vp_1' \to \neg \vp_1$ and $\vp_2' \to \neg \vp_2$ and $\pa $ proves $\neg \vp_1 \to\vp_1'$ and $\neg \vp_2 \to\vp_2'$.
By Lemma \ref{lem: conjunction of EPk formulas}, there exists $\vp'\in \EP_{k+1}$ such that $\FV{\vp'}=\FV{\vp_1' \land \vp_2'}$ and $\ha \vdash \vp' \lr \vp_1' \land \vp_2'$.
Then we have that $\ha+\LEM{\Sigma_k} $ proves
$$\vp' \lr \vp_1' \land \vp_2' \underset{\text{[I.H.] }\LEM{\Sigma_k}}{\lra} \neg \vp_1 \land \neg \vp_2 \lr \neg (\vp_1 \lor \vp_2)$$
and also $\pa $ proves the converse.

Case of $\vp:\equiv \vp_1 \land \vp_2$:
For item $1$, suppose $\vp_1 \land \vp_2 \in \U_{k+1}^+$.
Then $\vp_1 , \vp_2 \in \U_{k+1}^+$.
By using the induction hypothesis and Lemma \ref{lem: conjunction of EPk formulas}, one can take a witness for $\vp_1 \land \vp_2$ in a straightforward way.
Item $2$ follows from the induction hypothesis as in the case of $\vp:\equiv \vp_1 \lor \vp_2$: $\vp_1' \lor \vp_2' \in \EP_{k+1}$ is the witness since $\ha +\LEM{\Sigma_k}$ proves
$$\vp_1' \lor \vp_2'  \underset{\text{[I.H.] }\LEM{\Sigma_k}}{\lra} \neg \vp_1 \lor \neg \vp_2 \to \neg (\vp_1 \land \vp_2)$$
and $\pa$ proves the converse.

Case of $\vp:\equiv \vp_1 \to \vp_2$:
For item $1$, suppose $\vp_1 \to \vp_2 \in \U_{k+1}^+$.
Then $\vp_1 \in \E_{k+1}^+$ and $\vp_2 \in \U_{k+1}^+$.
By the induction hypothesis, there exist $\vp_1', \vp_2'\in \EP_{k+1}$ such that
$\ha+\LEM{\Sigma_k} $ proves $\vp_1' \to \neg \vp_1$ and $\vp_2' \to  \vp_2$ and $\pa $ proves $\neg \vp_1 \to\vp_1'$ and $ \vp_2 \to \vp_2'$.
Now $\vp_1' \lor \vp_2' \in \EP_{k+1}$ and $\ha +\LEM{\Sigma_k}$ proves
$$
\vp_1' \lor \vp_2'  \underset{\text{[I.H.] }\LEM{\Sigma_k}}{\lra} \neg \vp_1 \lor \vp_2 \to (\vp_1 \to \vp_2). 
$$
On the other hand, $\pa$ proves the converse.
For item $2$, suppose $\vp_1 \to \vp_2 \in \E_{k+1}^+$.
Then $\vp_1 \in \U_{k+1}^+$ and $\vp_2 \in \E_{k+1}^+$.
By the induction hypothesis, there exist $\vp_1', \vp_2'\in \EP_{k+1}$ such that
$\ha+\LEM{\Sigma_k} $ proves $\vp_1' \to \vp_1$ and $\vp_2' \to \neg \vp_2$ and $\pa $ proves $\vp_1 \to \vp_1'$ and $ \neg \vp_2 \to \vp_2'$.
By Lemma \ref{lem: conjunction of EPk formulas}, there exists $\vp'\in \EP_{k+1}$ such that $\FV{\vp'}=\FV{\vp_1' \land \vp_2'}$ and $\ha \vdash \vp' \lr \vp_1' \land \vp_2'$.
Then we have that $\ha +\LEM{\Sigma_k}$ proves
$$
\vp' \lr \vp_1' \land \vp_2' \underset{\text{[I.H.] }\LEM{\Sigma_k}}{\lra} \vp_1 \land \neg \vp_2 \to \neg (\vp_1 \to \vp_2) 
$$
and also that $\pa$ proves the converse.

Case of $\vp:\equiv \exists x \vp_1$:
For item $1$, suppose $\exists x \vp_1 \in \U_{k+1}^+$.
Then $\exists x \vp_1 \in \E_{k}^+$.
By Remark \ref{rem: Sk-LEM -> Sk-DNE, Uk-DNS and PkvPk-DNE}, there exists $\vp'\in \Sigma_k$ such that $\FV{\vp'}=\FV{\vp}$ and $\ha + \LEM{\Sigma_k} \vdash \vp' \lr \vp$.
Since $\Sigma_k$ can be seen as a subclass of $\Pi_{k+1}$, we are done.
For item $2$, suppose $\exists x \vp_1 \in \E_{k+1}^+$.
Then $\vp_1 \in \E_{k+1}^+$.
By the induction hypothesis, there exists $\vp_1' \in \EP_{k+1}$ such that
$\FV{\vp_1'}=\FV{\vp_1} $, $\ha+\LEM{\Sigma_k} \vdash \vp_1' \to \neg  \vp_1$ and $\pa \vdash \neg \vp_1 \to \vp_1'$.
Now $\forall x \vp_1'\in \EP_{k+1}$ and $\FV{\forall x\vp_1'}=\FV{\exists x \vp_1}$.
Then we have that $\ha +\LEM{\Sigma_k}$ proves
$$
\forall x \vp_1'  \underset{\text{[I.H.] }\LEM{\Sigma_k}}{\lra} \forall x \neg \vp_1
\lr \neg \exists x \vp_1
$$
and also that $\pa$ proves the converse.

Case of $\vp:\equiv \forall x \vp_1$:
For item $1$, suppose $\forall x \vp_1 \in \U_{k+1}^+$.
Then $ \vp_1 \in \U_{k+1}^+$.
By the induction hypothesis, there exists $\vp_1' \in \EP_{k+1}$ such that
$\FV{\vp_1'}=\FV{\vp_1} $, $\ha+\LEM{\Sigma_k} \vdash \vp_1' \to  \vp_1$ and $\pa \vdash \vp_1 \to \vp_1'$.
It is straightforward to see that $\forall x \vp_1'\in \EP_{k+1}$ is a witness for $\forall x \vp_1 \in \U_{k+1}^+$.
For item $2$, suppose $\forall x \vp_1 \in \E_{k+1}^+$.
Then $\forall x \vp_1 \in \U_{k}^+$.
By Remark \ref{rem: Sk-LEM -> Sk-DNE, Uk-DNS and PkvPk-DNE}, there exists $\vp'\in \Pi_k$ such that $\FV{\vp'}=\FV{\vp}$ and $\ha + \LEM{\Sigma_k} \vdash \vp' \lr \vp$.
Since $\neg \vp'$ is equivalent to some $\vp'' \in \Sigma_k$ in the presence of $\DNE{\Sigma_k}$ (cf. Remark \ref{rem: duals}),
we are done.
\end{proof}

\begin{lemma}
\label{lem: EPk-cons => Ukp-cons}
Let
$\T$ be a theory containing $\ha$ and $X$ be a set of $\ha$-sentences.
If $\pa + X$ is $\EP_{k+1}$-conservative over $\T +X $, then so is $\U_{k+1}$-conservative.
\end{lemma}
\begin{proof}
Let $\vp\in \U_{k+1}$.
Suppose $\pa + X \vdash \vp$.
By Lemma \ref{lem: U_k and EPk}, there exists $\vp'\in \EP_{k+1}$ such that $\FV{\vp} = \FV{\vp'}$, $\ha+\LEM{\Sigma_k} \vdash \vp' \to \vp$ and $\pa \vdash \vp \to\vp'$.
Then $\pa + X \vdash \vp'$.
By our assumption, we have $\T + X \vdash \vp'$.
As in the proof of Lemma \ref{lem: SkvPik-cons->Sk-LEM}, one can show $\T + X \vdash \LEM{\Sigma_k}$ by using the $\EP_{k+1}$-conservativity.
Then $\T + X \vdash \vp$ follows.
\end{proof}

\begin{theorem}
\label{thm: equivalents of Uk+1p-cons}
Let 
$\T$ be semi-classical arithmetic
and $X$ be a set of $\ha$-sentences in $\Q_{k+1}$.
Then the following are pairwise equivalent$:$
\begin{enumerate}
    \item
    \label{item: PA+X is Uk+1-cons. over T+X}
$\pa +X$ is $\U_{k+1}$-conservative over $\T +X ;$
    \item 
        \label{item: PA+X is EPk+1-cons. over T+X}
$\pa +X$ is $\EP_{k+1}$-conservative over $\T +X ;$
\item
    \label{item: T+X satisfies EPk+1-DNER}
$\T +X$ is closed under $\DNER{\EP_{k+1}} ;$ 
\item
    \label{item: T+X satisfies EPk+1-CDR}
$\T +X$ is closed under $\CDR{\EP_{k+1}} ;$
\item
    \label{item: T+X satisfies Uk+1-DNER}
$\T +X$ is closed under $\DNER{\U_{k+1}} ;$ 
\item
    \label{item: T+X satisfies Uk+1-CDR}
$\T +X$ is closed under $\CDR{\U_{k+1}} $.
\end{enumerate}
\end{theorem}
\begin{proof}
Since $\EP_{k+1} \subseteq \U_{k+1}$, the equivalence between \eqref{item: PA+X is Uk+1-cons. over T+X} and \eqref{item: PA+X is EPk+1-cons. over T+X} follows immediately from Lemma \ref{lem: EPk-cons => Ukp-cons}.

\noindent
$(\ref{item: PA+X is EPk+1-cons. over T+X}\to \ref{item: T+X satisfies EPk+1-DNER}):$
Let $\vp\in \EP_{k+1}$ and assume $\T + X \vdash \neg \neg \vp$.
Since $\T + X \subseteq \pa + X$, we have $\pa + X \vdash  \vp$.
By $\eqref{item: PA+X is EPk+1-cons. over T+X}$, we have $\T + X \vdash  \vp$.

\noindent
$(\ref{item: T+X satisfies EPk+1-DNER} \to \ref{item: T+X satisfies EPk+1-CDR}):$
Let $\vp, \psi(x)\in \EP_{k+1}$ and $x \notin \FV{\vp}$.
Assume $\T +X \vdash \forall x (\vp \lor \psi(x))$.
Since $\ha$ proves  $\neg \neg (\vp \lor \neg \vp)$ and $(\vp \lor \neg \vp)\land \forall x (\vp \lor  \psi(x)) \to \vp \lor \forall x \psi(x)$, we have $\T +X \vdash \neg \neg (\vp \lor \forall x \psi(x))$.
Since $\vp \lor \forall x \psi(x) \in \EP_{k+1}$, by $\DNER{\EP_{k+1}}$, we have  $\T +X \vdash \vp \lor \forall x \psi(x)$.

\noindent
$(\ref{item: T+X satisfies EPk+1-CDR} \to \ref{item: PA+X is EPk+1-cons. over T+X}):$
Assume that $\T +X$ is closed under $\CDR{\EP_{k+1}} $.
By Lemma \ref{lem: Pk+1-CDR => Sk-LEM}, we have $\T+X \vdash \LEM{\Sigma_k}$.
We show that $\pa + X \vdash \vp_1 \lor \dots \lor \vp_n$ implies $\T + X \vdash \vp_1 \lor \dots \lor \vp_n$ for any $\vp_1, \dots, \vp_n \in \EP_{k+1}$ by induction on the sum of the complexity of  $\vp_1, \dots, \vp_n \in \EP_{k+1}$.

First, suppose that all of $\vp_1, \dots, \vp_n$ are in $\Pi_{k+1}$.
Let $\vp_i:\equiv \forall x_i \vp_{i}'$ with $\vp_i'\in \Sigma_k$ for each $i\in \{ 1,\dots, n\}$.
Assume $\pa +X \vdash \vp_1 \lor \dots \lor \vp_n$.
Then $\pa +X \vdash \vp_1' \lor \dots \lor \vp_n'$.
Since $\T +X \vdash \LEM{\Sigma_k}$ and $X\subseteq \Q_{k+1}$, by Corollary \ref{cor: PA is Ek+1-cons. over HA+SkLEM}, 
we have $\T +X \vdash \vp_1' \lor \dots \lor \vp_n'$.
Then $\T +X \vdash  \forall x_1 (\vp_1' \lor \dots \lor \vp_n')$ follows.
By $\CDR{\EP_{k+1}}$, we have $\T +X \vdash  \forall x_1 \vp_1' \lor \vp_2' \lor \dots \lor \vp_n'$.
Iterating this procedure for more $n-1$ times, we have $\T +X \vdash  \forall x_1 \vp_1' \lor \dots \lor \forall x_n \vp_n'$.

Secondly, suppose $\vp_1, \dots, \vp_n \in \EP_{k+1}$ and $\vp_n:\equiv \vp_n' \lor \vp_n''$ with $\vp_n', \vp_n''\in \EP_{k+1}$.
Without loss of generality, let $n >1$.
Assume $\pa +X \vdash \vp_1 \lor \dots  \lor \vp_{n-1}  \lor \vp_n$, equivalently, $\pa +X \vdash \vp_1 \lor \dots\lor \vp_{n-1} \lor \vp_n' \lor \vp_n''$.
By the induction hypothesis, we have $\T +X \vdash \vp_1 \lor \dots \lor \vp_{n-1} \lor \vp_n' \lor \vp_n''$, equivalently, $\T +X \vdash \vp_1 \lor \dots  \lor \vp_{n-1}  \lor \vp_n$.

Finally, suppose $\vp_1, \dots, \vp_n \in \EP_{k+1}$ and $\vp_n: \equiv \forall x_n \vp_n'$ with $\vp_n' \in \EP_{k+1}$.
Without loss of generality, let $n >1$.
Assume $\pa +X \vdash \vp_1 \lor \dots  \lor \vp_{n-1}  \lor \vp_n$.
Then $\pa +X \vdash \vp_1 \lor \dots \lor \vp_{n-1} \lor \vp_n'$ follows.
By the induction hypothesis, we have $\T +X \vdash \vp_1 \lor \dots \lor \vp_{n-1} \lor \vp_n'$, and hence, $\T +X \vdash \forall x_n( \vp_1 \lor \dots\lor \vp_{n-1} \lor \vp_n')$.
By $\CDR{\EP_{k+1}}$, we have $\T +X \vdash \vp_1 \lor \dots \lor \vp_{n-1} \lor \vp_n$.

The implications $(\ref{item: PA+X is Uk+1-cons. over T+X}\to \ref{item: T+X satisfies Uk+1-DNER})$ and $(\ref{item: T+X satisfies Uk+1-DNER} \to \ref{item: T+X satisfies Uk+1-CDR})$ are shown as for
$(\ref{item: PA+X is EPk+1-cons. over T+X}\to \ref{item: T+X satisfies EPk+1-DNER})$ and $(\ref{item: T+X satisfies EPk+1-DNER} \to \ref{item: T+X satisfies EPk+1-CDR})$ respectively.
In addition, $(\ref{item: T+X satisfies Uk+1-CDR} \to \ref{item: T+X satisfies EPk+1-CDR})$ is trivial.
\end{proof}

Next, we characterize the $(\Pi_{k+1}\lor \Pi_{k+1})$-conservativity by several rules.

\begin{lemma}
\label{lem: PkvPk-DNER => Pk-CDR}
Let
$\T$ be a theory containing $\ha$.
If $\T$ is closed under $\DNER{(\Pi_k \lor \Pi_k)}$, then so is  $\CDR{\Pi_{k}}$.
\end{lemma}
\begin{proof}
The proof of $(\ref{item: T+X satisfies EPk+1-DNER} \to \ref{item: T+X satisfies EPk+1-CDR})$ of Theorem \ref{thm: equivalents of Uk+1p-cons} works.
\end{proof}

\begin{lemma}
\label{lem: Sk-DMLDR <=> PkvPk-DNER}
Let  $\T$ be a theory containing $\ha$.
Then  $\T$ is closed under $\DMLDR{\Sigma_{k}}$ if and only if $\T$ is closed under $\DNER{\left(\Pi_{k} \lor \Pi_{k}\right)}$.
\end{lemma}
\begin{proof}
One can show the ``only if'' direction as in the proof of  that in Lemma \ref{lem: Pk-DMLDR <=> Sk-DNER}.
For the converse direction, again by the corresponding proof in  Lemma \ref{lem: Pk-DMLDR <=> Sk-DNER}, it suffices to show that if $\T$ is closed under $\DNER{\left(\Pi_{k} \lor \Pi_{k}\right)}$, then $T $ proves $\LEM{\Sigma_{k-1}}$.
The latter is the case by Lemmata \ref{lem: PkvPk-DNER => Pk-CDR} and \ref{lem: Pk+1-CDR => Sk-LEM}.
\end{proof}

\begin{theorem}
\label{thm: equivalents of (Pk+1 v Pk+1)-cons}
Let
$\T$ be semi-classical arithmetic
and $X$ be a set of $\ha$-sentences in $\Q_{k+1}$.
Then the following are pairwise equivalent$:$
\begin{enumerate}
    \item
    \label{item: PA+X is (Pk+1 v Pk+1)-cons. over T+X}
$\pa +X$ is $\left(\Pi_{k+1} \lor \Pi_{k+1} \right)$-conservative over $\T +X ;$
\item
    \label{item: T+X satisfies Pk+1-DNER}
$\T +X$ is closed under $\DNER{ (\Pi_{k+1} \lor \Pi_{k+1})} ;$ 
\item
    \label{item: T+X satisfies Pk+1-CDR}
$\T +X$ is closed under $\CDR{ \Pi_{k+1}} ;$
\item
    \label{item: T+X satisfies Sk+1-DMLDR}
$\T +X$ is closed under $\DMLDR{ \Sigma_{k+1}} ;$
\item
    \label{item: T+X satisfies Sk+1-DMLR and T+X |- Sk-DNE}
$\T +X$ is closed under $\DMLR{ \Sigma_{k+1}}$ and $T+X$ proves $\DNE{\Sigma_{k}}$.
\end{enumerate}
\end{theorem}
\begin{proof}
One can show $(\ref{item: PA+X is (Pk+1 v Pk+1)-cons. over T+X} \to \ref{item: T+X satisfies Pk+1-DNER})$ as in the proof of $(\ref{item: PA+X is EPk+1-cons. over T+X}\to \ref{item: T+X satisfies EPk+1-DNER} )$ of Theorem \ref{thm: equivalents of Uk+1p-cons}.
The implication $(\ref{item: T+X satisfies Pk+1-DNER} \to \ref{item: T+X satisfies Pk+1-CDR})$ is by Lemma \ref{lem: PkvPk-DNER => Pk-CDR}.

We show $(\ref{item: T+X satisfies Pk+1-CDR} \to \ref{item: PA+X is (Pk+1 v Pk+1)-cons. over T+X} )$.
Assume that $\T +X$ is closed under $\CDR{ \Pi_{k+1}} $.
By Lemma \ref{lem: Pk+1-CDR => Sk-LEM}, we have $\T + X \vdash \LEM{\Sigma_k}$.
Let $\vp_1:\equiv \forall x_1 \psi_1$ and  $\vp_2: \equiv \forall x_2 \psi_2$ with $\psi_1, \psi_2 \in \Sigma_{k}$.
Suppose $\pa + X \vdash \forall x_1\psi_1 \lor \forall x_2\psi_2$.
Then $\pa + X \vdash  \neg \left( \exists x_1 \psi_1^\perp \land \exists x_2 \psi_2^\perp  \right)$.
Since $\exists x_1 \psi_1^\perp \land \exists _2 \psi_2^\perp$ is equivalent to a formula in $\Sigma_{k+1}$ (cf. \cite[Lemma 4.3.(2)]{FK20-1}),
by Remark \ref{rem: duals},
there exists $\xi\in \Pi_{k+1}$ such that $\FV{\xi} =\FV{\forall x_1 \psi_1 \lor \forall x \psi_2}$ and $\ha + \DNE{\Sigma_k} \vdash \xi \lr \neg \left( \exists x_1 \psi_1^\perp \land \exists x_2 \psi_2^\perp  \right)$.
Then we have $\pa + X \vdash \xi$.
Since $X\subseteq \Q_{k+1}$, by Corollary \ref{cor: Conservation result with X}, we have  $\ha + X + \LEM{\Sigma_{k}} \vdash \xi$.
Since $\DNE{\Sigma_k}$ is derivable from $\LEM{\Sigma_k}$, we have  that $T + X $ proves  $\neg \left( \exists x_1 \psi_1^\perp \land \exists x_2 \psi_2^\perp  \right)$, equivalently, $\forall x_1, x_2 \neg \left( \psi_1^\perp \land \psi_2^\perp  \right)$.
Since $\T  + X \vdash \DNE{\Sigma_k}$, again by Remark \ref{rem: duals}, we have that $T + X $ proves 
$\forall x_1, x_2 \neg \left( \neg \psi_1\land \neg \psi_2  \right)$, equivalently, $ \forall x_1, x_2 \neg \neg (\psi_1 \lor \psi_2)$.
Since $\psi_1 \lor \psi_2$ is equivalent to a formula in $\Sigma_k$ (cf. \cite[Lemma 4.4]{FK20-1}),
$\T + X \vdash \forall x_1, x_2 (\psi_1 \lor \psi_2)$ follows.
By using $\CDR{\Pi_{k+1}}$ twice, we have $\T +X \vdash \forall x_1\psi_1 \lor \forall x_2\psi_2$.

The equivalence $(\ref{item: T+X satisfies Pk+1-DNER} \lr \ref{item: T+X satisfies Sk+1-DMLDR})$ is by Lemma \ref{lem: Sk-DMLDR <=> PkvPk-DNER}.
The implication $(\ref{item: T+X satisfies Sk+1-DMLR and T+X |- Sk-DNE} \to \ref{item: T+X satisfies Sk+1-DMLDR} )$ is by the fact that for $\vp\in \Sigma_{k+1}$, $\vp^\perp$ is derived from $\neg \vp$ in the presence of $\DNE{\Sigma_k}$ (cf. Remark \ref{rem: duals}).
The implication $( \ref{item: T+X satisfies Pk+1-CDR} \, \& \, \ref{item: T+X satisfies Sk+1-DMLDR}  \to \ref{item: T+X satisfies Sk+1-DMLR and T+X |- Sk-DNE})$ is by
Lemma \ref{lem: Pk+1-CDR => Sk-LEM} (note that $\LEM{\Sigma_k}$ implies $\DNE{\Sigma_k}$).
\end{proof}

\begin{remark}
\label{rem: equivalents of VPk+1-cons}
From the perspective of Remark \ref{rem: class hierarchy between Pi_k and ESk+1}, it is natural to ask the status of the $\bigvee \Pi_{k+1}$-conservativity.
As in the proof of Theorem \ref{thm: equivalents of (Pk+1 v Pk+1)-cons}, one can show the following equivalence:
\begin{enumerate}
    \item
    \label{item: PA+X is VPk+1-cons. over T+X}
$\pa +X$ is $\bigvee \Pi_{k+1}$-conservative over $\T +X ;$
\item
    \label{item: T+X satisfies VPk+1-DNER}
For any $\vp_1 , \dots, \vp_n\in \Pi_{k+1}$, if $\T + X \vdash \neg\neg (\vp_1 \lor \dots \lor  \vp_n)$, then $\T + X \vdash \vp_1 \lor \dots \lor  \vp_n;$
\item
    \label{item: T+X satisfies VPk+1-CDR}
For any $\vp_1 , \dots, \vp_n\in \Pi_{k+1}$ such that $x \notin \FV{\vp_1 \lor \dots \lor \vp_{n-1}}$, if $\T + X \vdash \forall x (\vp_1 \lor \dots \lor \vp_{n-1} \lor \vp_n)$, then $\T + X \vdash  \vp_1 \lor \dots \lor \vp_{n-1} \lor \forall x \vp_n;$
\item
For any $\vp_1 , \dots, \vp_n\in \Sigma_{k+1}$, if $\T + X \vdash \neg (\vp_1 \land \dots \land  \vp_n)$, then $\T + X \vdash \vp_1^\perp \lor \dots \lor  \vp_n^\perp ;$
\item
$\T +X$ proves $\DNE{\Sigma_k}$ and for any $\vp_1 , \dots, \vp_n\in \Sigma_{k+1}$, if $\T + X \vdash \neg (\vp_1 \land \dots \land  \vp_n)$, then $\T + X \vdash \neg \vp_1 \lor \dots \lor  \neg \vp_n;$
\end{enumerate}
where $X\subseteq \Q_{k+1}$.
This characterization suggests that the $\bigvee \Pi_{k+1}$-conservativity lies strictly between the $\U_{k+1}$-conservativity and the $\left(\Pi_{k+1}\lor \Pi_{k+1}\right)$-conservativity, but we do not have the proof of the strictness.
\end{remark}

\begin{remark}
From the comparison between \cite[Corollary 7.6]{FK20-2} and the equivalences in Theorem \ref{thm: equivalents of (Pk+1 v Pk+1)-cons}, it is natural to ask whether the (contrapositive) collection rule restricted to formulas in $\Pi_{k+1}$ is also equivalent to the items in Theorem \ref{thm: equivalents of (Pk+1 v Pk+1)-cons}.
This question is still open.
\end{remark}

\section{Conservation theorems for the classes of sentences}
\label{sec: CONS for classes of sentences}
In the study of fragments of $\pa$, the conservativity for classes of sentences has been studied extensively e.g. in \cite[Section 2]{KP88}.
The following proposition states that  the conservativity for a class of formulas is equivalent to that restricted only to sentences if the class is closed under taking a universal closure:
\begin{proposition}
\label{prop: utcons. => cons. for forall closed classes}
Let $\Gamma$ be a class of $\ha$-formulas such that $\Gamma$ is closed under taking a universal closure.
For any theories $\T$ and $\T'$ containing $\ha$ in the language of $\ha$, if $\T' $ is conservative over $\T$ for any sentences in $\Gamma$, then $\T' $ is $\Gamma$-conservative over $\T$.
\end{proposition}
\begin{proof}
Let $\vp \in \Gamma$.
Assume $\T' \vdash \vp$.
Then we have $\T' \vdash \wt{\vp}$ where $\wt{\vp}$ is the universal closure of $\vp$.
Since $\wt{\vp}$ is a sentence in $\Gamma$, by our assumption, we have  $\T \vdash \wt{\vp}$, and hence,  $\T \vdash \vp$.
\end{proof}
Therefore, for classes as $\Pi_{k}, \U_{k}, \EP_{k}$ etc., the strength of the conservativity does not vary even if we restrict them only to sentences.
On the other hand, since $\Sigma_{k}, \E_{k}, \F_{k}$ etc. are not closed under taking a universal closure, this is not the case for such classes.
In what follows, we explore the relation on the notion that $\pa$ is $\Gamma$-conservative over $\T$ for semi-classical arithmetic $\T$ and the class $\Gamma$ of sentences.

\begin{definition}
\label{def: class of sentences}
For a class $\Gamma$ of $\ha$-formulas, $ \ut{\Gamma}$ denotes the class of $\ha$-sentences in $\Gamma$.
\end{definition}

\subsection{Conservation theorems for $\Sigma_k$ sentences and $\E_k$ sentences}
\label{sec: Cons. theorems for utS_k and utEk}
For the $\ut{\Sigma_{k}}$-conservativity, we have the following:
\begin{proposition}
\label{prop: utSk+1-cons <=> utSk+1-DNER}
Let
$\T$ be semi-classical arithmetic containing $\LEM{\Sigma_{k-1}}$, and $X$ be a set of $\ha$-sentences in $\Q_{k}$.
Then $\pa +X$ is $\ut{\Sigma_{k+1}}$-conservative over $\T+X $
if and only if
$\T+X$ is closed under $\DNER{\ut{\Sigma_{k+1}}}$.
\end{proposition}
\begin{proof}
We first show the ``only if'' direction.
Let $\vp\in \ut{\Sigma_{k+1}}$.
Assume $\T +X \vdash \neg \neg \vp$.
Then $\pa + X \vdash \vp$.
Since $\pa +X$ is now $\ut{\Sigma_{k+1}}$-conservative over $\T+X$, we have $\T + X \vdash \vp$.

In the following, we show the converse direction.
Without loss of generality, assume $k>0$.
Let $\exists x \forall y \,\psi \in \ut{\Sigma_{k+1}} $ with $\psi$ in $\Sigma_{k-1}$.
Assume $\pa +X \vdash \exists x \forall y \psi$.
By Proposition \ref{Soundness of *-translation}.\eqref{item: PA |- A => HA |- A*}, we have $\ha^{\ph} + X^{\ph}\vdash \neg_{\ph}\neg_{\ph} \exists x  \forall y \psi^{\ph}$, and hence, $\ha^{\ph} + \LEM{\Sigma_{k-1}} + X \vdash \neg_{\ph}\neg_{\ph} \exists x  \forall y \psi^{\ph}$ by Lemma \ref{lem: For A in Qk+1, HA + SigmakLEM |- A -> A*}.
Using Lemma \ref{Kurahashi Lemma: A* <-> A v *}.\eqref{item: A v * for Sigma_{k+1}}, we have
$\ha^{\ph} + \LEM{\Sigma_{k-1}} + X \vdash \neg_{\ph}\neg_{\ph} \exists x  \forall y \left(\psi\lor \ph\right)$.
By substituting $\ph$ with $\perp$ (cf. Lemma \ref{lem: Substitution}), we have $\ha + \LEM{\Sigma_{k-1}} + X \vdash
\neg\neg \exists x \forall y  \psi$.
Since $\T$ is semi-classical arithmetic containing $\LEM{\Sigma_{k-1}}$, we have
$\T +X \vdash \neg\neg \exists x \forall y  \psi$.
By $\DNER{\ut{\Sigma_{k+1}}}$, $\T +X \vdash \exists x \forall y  \psi$ follows.
\end{proof}

Proposition \ref{prop: utSk+1-cons <=> utSk+1-DNER} is a counterpart of the equivalence between \eqref{item: CONS(Sk+1)} and \eqref{item: Sk+1DNER} in Theorem \ref{thm: conservation theorems equivalent to SkLEM} for the case of sentences.
In what follows, we deal with the $\ut{\E_{k+1}}$-conservativity.
In particular, we show that the $\ut{\E_{k+1}}$-conservativity can be reduced to $\ut{\ES_{k+1}}$-conservativity.

\begin{lemma}
\label{lem: conjunction and disjunction of ESk+1 formulas}
For $\ha$-formulas $\vp_1, \vp_2\in \ES_{k+1}$, there exist $\psi, \xi\in \ES_{k+1}$ such that $\FV{\psi} =\FV{\vp_1 \land \vp_2} = \FV{\vp_1 \lor \vp_2} = \FV{\xi}$ and $\ha$ proves $\psi \lr \vp_1 \land \vp_2$ and   $\xi \lr \vp_1 \lor \vp_2$.
\end{lemma}
\begin{proof}
Let $\vp_1:\equiv \exists x_1, \dots, x_n \vp_1'$ and $\vp_2:\equiv \exists y_1, \dots, y_m \vp_2'$ with $\vp_1', \vp_2' \in \EP_{k}$.
Without loss of generality, assume $x_1, \dots, x_n \notin \FV{\vp_2'} $ and  $y_1, \dots, y_m \notin \FV{\vp_1'}$.

By Lemma \ref{lem: conjunction of EPk formulas}, there exists $\psi' \in \EP_{k}$ such that $\FV{\psi'}=\FV{\vp_1' \land \vp_2'}$ and $\ha \vdash \psi' \lr \vp_1' \land \vp_2'$.
Put $\psi :\equiv \exists  x_1, \dots, x_n,  y_1, \dots, y_m  \, \psi'$, which is in $ \ES_{k+1}$.
Then it is trivial that $\FV{\psi} = \FV{\vp_1 \land \vp_2}$ and $\ha \vdash \psi \lr \vp_1 \land \vp_2$. 

Put $\xi:\equiv \exists  x_1, \dots, x_n,  y_1, \dots, y_m  \left(\vp_1' \lor \vp_2'\right)$, which is in $ \ES_{k+1}$.
Since $\xi$ is equivalent to $\exists  x_1, \dots, x_n  \vp_1' \lor \exists  y_1, \dots, y_m \vp_2'$ over $\ha$, we have that $\FV{\xi} = \FV{\vp_1 \lor \vp_2}$ and $\ha \vdash \xi \lr \vp_1 \lor \vp_2$.
\end{proof}

\begin{lemma}
\label{lem: E_k+1 and ESk+1}
For a $\ha$-formula $\vp$, the following hold$:$
\begin{enumerate}
    \item 
    If $\vp \in \U_{k+1}^+$, then there exists $\vp'\in \ES_{k+1}$ such that $\FV{\vp} = \FV{\vp'}$, $\ha+\LEM{\Sigma_{k-1}} \vdash \vp' \to \neg \vp$ and $\pa \vdash \neg \vp \to\vp'$.
    \item
    If $\vp \in \E_{k+1}^+$, then there exists $\vp'\in \ES_{k+1}$ such that $\FV{\vp} = \FV{\vp'}$, $\ha+\LEM{\Sigma_{k-1}} \vdash \vp' \to  \vp$ and $\pa \vdash \vp \to\vp'$.
\end{enumerate}
\end{lemma}
\begin{proof}
Note that $\U_{k+1}^+ = \R_{k+1}'$ and $\E_{k+1}^+ =\J_{k+1}'$ where $\R_{k+1}'$ and $\J_{k+1}'$ are the classes defined in Remark \ref{rem: Rk' and Jk'}.
Then it suffices to show items $1$ and $2$ where $\U_{k+1}^+$ and $\E_{k+1}^+$ are replaced by $\R_{k+1}'$ and $\J_{k+1}'$ respectively.
In the following, we show the assertions by induction on the constructions of $\R_{k+1}'$ and $\J_{k+1}'$.

For $\vp\in \E_k^+ \subseteq \R'_{k+1}$, by Lemma \ref{lem: U_k and EPk}, there exists $\vp'\in \EP_k (\subseteq \ES_{k+1})$  such that $\ha +\LEM{\Sigma_{k-1}}  \vdash \vp' \to \neg \vp$ and $\pa \vdash \neg\vp \to \vp'$. 
For $\vp\in \U_k^+ \subseteq \J'_{k+1}$, by Lemma \ref{lem: U_k and EPk}, there exists $\vp'\in \EP_k (\subseteq \ES_{k+1})$  such that $\ha +\LEM{\Sigma_{k-1}}  \vdash \vp' \to  \vp$ and $\pa \vdash \vp \to \vp'$. 
For the induction step, let $\vp_1, \vp_2\in \R_{k+1}'$ and $\psi_1, \psi_2\in \J_{k+1}'$ and  $\vp_1', \vp_2',\psi_1',\psi_2' \in \ES_{k+1}$ satisfy $\FV{\vp_1}=\FV{\vp_1'}$, $\FV{\vp_2}=\FV{\vp_2'}$, $\FV{\psi_1}=\FV{\psi_1'}$, $\FV{\psi_2}=\FV{\psi_2'}$ and that
$\ha +\LEM{\Sigma_{k-1}} $ proves  $ \vp_1' \to \neg \vp_1$, $ \vp_2' \to \neg \vp_2$, $ \psi_1' \to \psi_1$, $ \psi_2' \to \psi_2$ and $\pa$ proves $\neg\vp_1 \to \vp_1'$, $\neg\vp_2 \to \vp_2'$, $\psi_1 \to \psi_1'$, $\psi_2 \to \psi_2'$. 
By Lemma \ref{lem: conjunction and disjunction of ESk+1 formulas}, for any conjunction and disjunction of $\vp_1', \vp_2',\psi_1',\psi_2' \in \ES_{k+1}$, there exists an equivalent (over $\ha$) $\xi \in \ES_{k+1}$ which preserves the free variables.
For $\vp:\equiv \vp_1 \lor \vp_2 \in \R_{k+1}'$, take $\vp'\in \ES_{k+1}$ as an equivalent of $\vp_1' \land \vp_2'$.
For $\vp:\equiv \psi_1 \lor \psi_2 \in \J_{k+1}'$, take $\vp' \in \ES_{k+1}$ as an equivalent of $\psi_1' \lor \psi_2'$.
For $\vp:\equiv \vp_1 \land \vp_2 \in \R_{k+1}'$, take $\vp' \in \ES_{k+1}$ as an equivalent of $ \vp_1' \lor \vp_2'$.
For $\vp:\equiv \psi_1 \land \psi_2 \in \J_{k+1}'$, take $\vp'\in \ES_{k+1}$ as an equivalent of $\psi_1' \land \psi_2'$.
For $\vp:\equiv \psi_1 \to \vp_2 \in \R_{k+1}'$, take $\vp'\in \ES_{k+1}$ as an equivalent of $\psi_1' \land \vp_2'$.
For $\vp:\equiv \vp_1 \to \psi_2 \in \J_{k+1}'$, take $\vp'\in \ES_{k+1}$ as an equivalent of $\vp_1' \lor \psi_2'$.
For $\vp:\equiv \forall x \vp_1  \in \R_{k+1}'$, take $\vp':\equiv \exists x \vp_1'\in \ES_{k+1}$.
For $\vp:\equiv \exists x \psi_1  \in \J_{k+1}'$, take $\vp' :\equiv \exists x \psi_1' \in \ES_{k+1}$.
We leave the routine verification for the reader.
\end{proof}

\begin{corollary}
\label{cor: E_k+1 and Sk+1}
For a $\ha$-formula $\vp$, the following hold$:$
\begin{enumerate}
    \item 
    \label{item: Uk+1 in cor: E_k+1 and Sk+1}
    If $\vp \in \U_{k+1}^+$, then there exists $\vp'\in \Sigma_{k+1}$ such that $\FV{\vp} = \FV{\vp'}$, $\ha+\DNE{(\Pi_k \lor \Pi_k)} \vdash \vp' \to \neg \vp$ and $\pa \vdash \neg \vp \to\vp'$.
    \item
    If $\vp \in \E_{k+1}^+$, then there exists $\vp'\in \Sigma_{k+1}$ such that $\FV{\vp} = \FV{\vp'}$, $\ha+\DNE{(\Pi_k \lor \Pi_k)} \vdash \vp' \to  \vp$ and $\pa \vdash \vp \to\vp'$.
\end{enumerate}
\end{corollary}
\begin{proof}
Since $\vp'\in \ES_{k+1}$ is of the form $\exists x \vp_1'$ where $\vp_1'\in \EP_k \subseteq \U_k^+$, by Theorem \ref{thm: PNFT}.\eqref{eq: item for Pi_k in PNFT}, there exists $\psi\in \Sigma_{k+1}$ such that $\FV{\vp'}=\FV{\psi}$ and $\ha + \DNE{(\Pi_k \lor \Pi_k)} \vdash \vp' \lr \psi$.
Since $\ha + \DNE{(\Pi_k \lor \Pi_k)}$ proves $\LEM{\Sigma_{k-1}}$, our corollary follows from Lemma \ref{lem: E_k+1 and ESk+1}.
\end{proof}

\begin{theorem}
\label{thm: Ek+1p cons. <=> ESk+1-cons}
Let
$\T$ be semi-classical arithmetic 
and $X$ be a set of $\ha$-sentences.
Then $\pa + X$ is $\ut{\E_{k+1}}$-conservative over $\T + X$ if and only if $\pa + X$ is $\ut{\ES_{k+1}}$-conservative over $\T + X$.
\end{theorem}
\begin{proof}
The ``only if'' direction is trivial since $\ut{\ES_{k+1}}\subseteq \ut{\E_{k+1}}$.
We show the converse direction.
Let $\vp\in \ut{\E_{k+1}}$.
Assume $\pa + X \vdash \vp$.
By Lemma \ref{lem: E_k+1 and ESk+1}, there exists $\vp'\in \ut{\ES_{k+1}}$ such that  $\ha + \LEM{\Sigma_{k-1}}
\vdash \vp' \to\vp$ and $\pa \vdash \vp \to \vp'$.
Then $\pa + X \vdash \vp'$.
By our assumption, we have $\T + X \vdash \vp'$.
On the other hand, as in the proof of Lemma \ref{lem: SkvPik-cons->Sk-LEM}, one can show  $\T  +X \vdash \LEM{\Sigma_{k-1}}$ (note that $\ut{\EP_{k}}$ can be seen as a sub-class of $\ut{\ES_{k+1}}$ and the $\ut{\EP_{k}}$-conservativity implies the $\EP_{k}$-conservativity by Proposition \ref{prop: utcons. => cons. for forall closed classes}).
Then we have $\T+X \vdash \vp$.
\end{proof}

\subsection{Conservation theorem for $\F_k$ sentences}
Next, we characterize the $\ut{\F_{k}}$-conservativity.
To investigate the class $\F_{k}$, it is convenient to consider the following class:
\begin{definition}
Let $\B_k^+$ be the class of formulas which are constructed from formulas in $\E_k^+ \cup \U_k^+$ by using logical connectives $\land, \lor$ and $\to$.
Let $\LEM{\B_k^+}$ be ${\rm LEM}$ restricted to formulas in $\B_k^+$.
\end{definition}

\begin{proposition}
\label{prop: SkLEM<->Bk+LEM}
$\ha \vdash \LEM{\Sigma_k} \lr  \LEM{\B_k^+}$.
\end{proposition}
\begin{proof}
First, $\ha + \LEM{\B_k^+} \vdash \LEM{\Sigma_k} $ is trivial since $\Sigma_k \subseteq \E_k^+$.
We show the converse direction.
By Remark \ref{rem: Sk-LEM -> Sk-DNE, Uk-DNS and PkvPk-DNE}, inside $\ha+\LEM{\Sigma_k}$, one may assume that $\vp \in \B_k^+$ is constructed from formulas in $\Sigma_k \cup \Pi_k$ by using logical connectives $\land, \lor$ and $\to$.
Then we have $\ha + \LEM{\Sigma_k} \vdash \LEM{\B_k^+}$ in a straightforward way.
\end{proof}

\begin{proposition}
\label{prop: B_k+=F_k+}
$\B_k^+ = \F_k$.
\end{proposition}
\begin{proof}
Since $\F_k^+ = \F_k$ (cf. Remark \ref{rem: On the non-commutativity of Ek, Uk and Fk}), it suffices to show $\B_k^+ = \F_k^+$.

First, $\B_k^+ \subseteq \F_k^+$ is trivial since $\E_k^+ \subseteq \F_k^+$, $\U_k^+ \subseteq \F_k^+$ and the fact that $\F_k^+$ is closed under $\land, \lor$ and $\to$.

We show that $\vp\in \F_k^+$ implies $\vp\in \B_k^+$ for all $\ha$-formulas $\vp$ by induction on the structure of formulas.
If $\vp$ is prime, since $\vp\in \B_k^+$, then we are done.
For the induction step, assume that it holds for $\vp_1$ and $\vp_2$.
If $\vp_1 \land \vp_2 \in \F_k^+$, then $\vp_1, \vp_2 \in \F_k^+$ follows.
By the induction hypothesis, we have $\vp_1, \vp_2 \in \B_k^+$, and hence, $\vp_1 \land \vp_2 \in \B_k^+$.
The cases of $\vp_1 \lor \vp_2$ and $\vp_1 \to \vp_2 $ are similar. 
If $\forall x \vp_1\in \F_k^+$, by the definition, we have $\forall x \vp_1 \in \U_k^+$, and hence, $\forall x \vp_1 \in \B_k^+$.
The case of $\exists  x \vp_1\in \F_k^+$ is similar.
\end{proof}

\begin{corollary}[cf. {\cite[Corollary 2.8.(i)]{ABHK04}}]
\label{cor: SkLEM<->Fk+LEM}
$\ha \vdash \LEM{\Sigma_k} \lr  \LEM{\F_k}$.
\end{corollary}
\begin{proof}
Immediate from Propositions \ref{prop: SkLEM<->Bk+LEM} and \ref{prop: B_k+=F_k+}.
\end{proof}

\begin{remark}
By using Proposition \ref{prop: B_k+=F_k+} and Theorem \ref{thm: PNFT},
one can show the following:
If $\vp \in \F_k$, then $\ha^{\ph} + \LEM{\Sigma_k} \vdash \vp^{\ph} \leftrightarrow \vp \lor {\ph}$.
This is an extension of Lemma \ref{Kurahashi Lemma: A* <-> A v *}.
\end{remark}

\begin{lemma}
\label{lem: Bkp and EPk v Sk}
For all $\vp \in \B_k^+$, there exist $\vp'$ and $\vp''$ which are constructed from formulas in $\EP_k \bigcup \Sigma_k$ by using $\land$ and $\lor$ only, and satisfy $\FV{\vp'}=\FV{\vp''}=\FV{\vp}$, $\ha +\LEM{\Sigma_{k-1}}$ proves $\vp' \to \vp$ and $\vp'' \to \neg \vp$, and $\pa$ proves $\vp \to \vp'$ and $\neg \vp \to \vp''$.
\end{lemma}
\begin{proof}
By induction on the construction of $\B_k^+$.

For the base case, first assume $\vp\in \U_{k}^+$.
By Lemma \ref{lem: U_k and EPk}, there exists $\vp'\in \EP_k$ such that $\FV{\vp}=\FV{\vp'}$, $\ha + \LEM{\Sigma_{k-1}}\vdash \vp' \to \vp$ and $\pa \vdash \vp \to \vp'$.
By Corollary \ref{cor: E_k+1 and Sk+1}, there exists $\vp''\in \Sigma_k$ such that $\FV{\vp}=\FV{\vp''}$, $\ha + \LEM{\Sigma_{k-1}}\vdash \vp'' \to \neg \vp$ (cf. Remark \ref{rem: Sk-LEM -> Sk-DNE, Uk-DNS and PkvPk-DNE}) and $\pa \vdash \neg \vp \to \vp''$.
Next assume $\vp\in \E_{k}^+$.
By  Corollary \ref{cor: E_k+1 and Sk+1}, there exists $\vp'\in \Sigma_k$ such that $\FV{\vp}=\FV{\vp'}$, $\ha + \LEM{\Sigma_{k-1}}\vdash \vp' \to \vp$ and $\pa \vdash \vp \to \vp'$.
By Lemma \ref{lem: U_k and EPk}, there exists $\vp''\in \EP_k$ such that $\FV{\vp}=\FV{\vp''}$, $\ha + \LEM{\Sigma_{k-1}}\vdash \vp'' \to  \neg \vp$ and $\pa \vdash \neg\vp \to \vp''$.

For the induction step, let $\vp_1, \vp_2\in \B_k^+$ and $\vp_1', \vp_1'', \vp_2', \vp_2''$ constructed from formulas in $\EP_k \bigcup \Sigma_k$ by using $\land$ and $\lor$ only satisfy the following:
$\FV{\vp_1'}=\FV{\vp_1''}=\FV{\vp_1}$, $\FV{\vp_2'}=\FV{\vp_2''}=\FV{\vp_2}$, $\ha +\LEM{\Sigma_{k-1}}$ proves $\vp_1' \to \vp_1$, $\vp_2' \to \vp_2$, $\vp_1'' \to \neg \vp_1$, $\vp_2'' \to \neg \vp_2$ and $\pa$ proves $\vp_1 \to \vp_1'$,  $\vp_2 \to \vp_2'$, $\neg \vp_1 \to \vp_1''$, $\neg \vp_2 \to \vp_2''$.
For $\vp:\equiv \vp_1 \land \vp_2$, take $\vp':\equiv \vp_1'\land \vp_2'$ and $\vp'':\equiv \vp_1'' \lor \vp_2'' $.
For $\vp:\equiv \vp_1 \lor \vp_2$, take $\vp':\equiv \vp_1'\lor \vp_2'$ and $\vp'':\equiv \vp_1'' \land \vp_2'' $.
For $\vp:\equiv \vp_1 \to \vp_2$, take $\vp':\equiv \vp_1'' \lor \vp_2'$ and $\vp'':\equiv \vp_1' \land \vp_2'' $.
We leave the routine verification for the reader.
\end{proof}

\begin{theorem}
\label{thm: Fkp cons. <=> (Sk v EPk)-cons.}
Let
$\T$ be semi-classical arithmetic
and $X$ be a set of $\ha$-sentences.
Then $\pa + X$ is $\ut{\F_{k}}$-conservative over $\T + X$ if and only if $\pa + X$ is $\left(\ut{\Sigma_{k}}\lor \ut{\EP_{k}}\right)$-conservative over $\T + X$.
\end{theorem}
\begin{proof} 
The ``only if'' direction is trivial since $\ut{\Sigma_{k}}\lor \ut{\EP_{k}} \subseteq \ut{\F_{k}}$.
We show the converse direction.
Let $\vp\in \ut{\F_{k}}$.
Assume $\pa + X \vdash \vp$.
By Lemma \ref{lem: Bkp and EPk v Sk} and Proposition \ref{prop: B_k+=F_k+}, there exist $\vp'$ which is constructed from formulas in $\ut{\Sigma_{k}}\bigcup \ut{\EP_{k}}$ by using $\land$ and $\lor$ only, and satisfy $\ha +\LEM{\Sigma_{k-1}}\vdash \vp' \to \vp$ and $\pa\vdash \vp \to \vp'$.
Without loss of generality, one may assume that $\vp'$ is of conjunctive normal form such that each conjunct is a disjunction of sentences in $\ut{\Sigma_{k}}\bigcup \ut{\EP_{k}}$.
Since disjunction of sentences in $\ut{\Sigma_{k}}$ is equivalent to a sentence in $\ut{\Sigma_{k}}$ over $\ha$ and $\ut{\EP_{k}}$ is closed under $\lor$, each conjunct can be assumed to be of the form $\psi \lor \xi$ where $\psi\in \ut{\Sigma_k}$ and $\xi\in \ut{\EP_k}$.
Let $\vp':\equiv \bigwedge_{1\leq i\leq n} \left( \psi_i \lor \xi_i \right)$ where $\psi_i\in \ut{\Sigma_k}$ and $\xi_i\in \ut{\EP_k}$.
Since $\pa + X \vdash \vp'$,  by the $\left(\ut{\Sigma_{k}}\lor \ut{\EP_{k}}\right)$-conservativity, we have that $\T + X$ proves $ \psi_i \lor \xi_i$ for each $i$.
Then we have $\T + X \vdash \vp'$.
Since $\pa + X$ is now $\EP_{k}$-conservative over $\T + X$ (cf. Proposition \ref{prop: utcons. => cons. for forall closed classes}),
as in the proof of Lemma \ref{lem: SkvPik-cons->Sk-LEM},
we have $\T + X \vdash \LEM{\Sigma_{k-1}}$.
Then $\T + X \vdash \vp$ follows.
\end{proof}

In what follows, by further investigating the $\left(\ut{\Sigma_{k}}\lor \ut{\EP_{k}}\right)$-conservativity in Theorem \ref{thm: Fkp cons. <=> (Sk v EPk)-cons.}, we  give a characterization of the $\ut{\F_{k}}$-conservativity by axiom schemata.
\begin{definition}
\label{def: DNEC and DNSC}
Let $\Gamma$ be a class of $\ha$-formulas.
We introduce the following axiom schemata: 
\begin{itemize}
 \item
$\DNEC{\Gamma}:\, \wt{\neg\neg \vp}\to \wt{\vp};$
\item
$\DNSC{\Gamma}:\, \wt{\neg\neg \vp}\to \neg \neg \wt{\vp};$
\end{itemize}
where $\vp\in \Gamma$ and $\wt{\neg\neg \vp}$ and $\wt{\vp}$ are universal closures of $\neg \neg \vp$ and $\vp$ respectively.
\end{definition}

\begin{proposition}
\label{prop: decomposition of Gamma-DNEC}
Let $\Gamma$ be a class of $\ha$-formulas such that $\Gamma$ is closed under taking a universal closure.
Then $\DNEC{\Gamma}$ is equivalent to $\DNSC{\Gamma} + \DNE{\ut{\Gamma}}$ over $\ha$.
\end{proposition}
\begin{proof}
It is trivial that $\DNEC{\Gamma}$ implies  $\DNSC{\Gamma}$ and also $\DNE{\ut{\Gamma}}$.
We show $\ha + \DNSC{\Gamma} + \DNE{\ut{\Gamma}} \vdash \DNEC{\Gamma} $.
Let $\vp \in \Gamma$.
By $\DNSC{\Gamma}$, $\wt{\neg \neg \vp}$ implies $\neg \neg \wt{\vp}$.
Since $\wt{\vp}$ is now in $\Gamma$, by $\DNE{\ut{\Gamma}}$, $\neg \neg \wt{\vp}$ implies $\wt{\vp}$.
Thus we have $\ha + \DNSC{\Gamma} + \DNE{\ut{\Gamma}} \vdash \wt{\neg \neg \vp} \to \wt{\vp}.$
\end{proof}

\begin{lemma}
\label{lem: an equivalent of (utSk v utEPk)-cons.} 
Let
$\T$ be a theory containing $\ha$ and satisfying the deduction theorem, and $X$ be a set of $\ha$-sentences in $\Q_k$.
If $\T + X $ proves $\LEM{\ut{\Sigma_k}}$ and $\T +X$ is closed under $\DNER{\EP_k}$ with assumptions of sentences in $\Pi_k:$\\[5pt]
$\T +X \vdash \psi \to \neg \neg \vp$ implies $\T +X \vdash \psi \to \vp$ for all $\psi\in \ut{\Pi_k}$ and $\vp\in \EP_k$,\\[5pt]
then $\pa+X$ is $\left(\ut{\Sigma_{k}}\lor \ut{\EP_{k}}\right)$-conservative over $\T + X$.
\end{lemma}
\begin{proof}
Let $\vp \in \ut{\Sigma_k}$ and $\psi\in \ut{\EP_k}$.
Assume $\pa + X \vdash \vp \lor \psi$.
Since $\T$ satisfies the deduction theorem and $\vp^\perp \in \ut{\Pi_k}$, by our second assumption, we have that $\T + X + \vp^\perp$ is closed under $\DNER{\EP_k}$.
Since $\vp^\perp \in \Q_k$, by
Theorem \ref{thm: equivalents of Uk+1p-cons},
we have that $\pa +X+\vp^\perp$ is $\EP_k$-conservative over $\T + X + \vp^\perp$.
Since $\pa + X + \vp^\perp \vdash \psi$, we have $\T + X + \vp^\perp \vdash \psi$, and hence,
\begin{equation}
\label{eq: dvp -> psi}
    \T + X  \vdash \vp^\perp \to \psi
\end{equation}
by the deduction theorem.
In addition, by our second assumption and Theorem \ref{thm: equivalents of Uk+1p-cons}, we have that $\T + X$ is closed under $\CDR{\EP_k}$, and hence, $\T+ X \vdash \LEM{\Sigma_{k-1}}$ by Lemma \ref{lem: Pk+1-CDR => Sk-LEM}.
Then, by Remark \ref{rem: duals}, we have $\T+X \vdash \neg \vp \to \vp^\perp$, and hence, $\T + X  \vdash \neg \vp \to \psi$ by \eqref{eq: dvp -> psi}.
On the other hand, by our first assumption,
we have
 $   \T+X \vdash \vp \lor \neg \vp$.
Then $\T+X \vdash \vp \lor \psi$ follows.
\end{proof}

\begin{theorem}
\label{thm: axiom schemata equivalent to Fkp-conservativity}
Let
$\T$ be semi-classical arithmetic
satisfying the deduction theorem
and $X$ be a set of $\ha$-sentences in $\Q_{k}$.
Then the following are pairwise equivalent$:$
\begin{enumerate}
    \item
    \label{item: PA+X is Fkp-cons. over T+X}
    $\pa + X$ is $\ut{\F_{k}}$-conservative over $\T + X;$
\item
\label{item: T+X |- utFkp-LEM + Ukp-DNSC}
    $\T + X $ proves $\LEM{\ut{\F_k}}$ and $\DNSC{\U_k};$
    \item
\label{item: T+X |- utFkp-LEM + Ukp-DNEC}
    $\T + X $ proves $\LEM{\ut{\Sigma_k}}$ and $\DNEC{\U_k};$
\item
\label{item: T+X |- utSk-LEM + EPk-DNEC}
    $\T + X $ proves $\LEM{\ut{\Sigma_k}}$ and $\DNEC{\EP_k}$.
\end{enumerate}
\end{theorem}
\begin{proof}
$(\ref{item: PA+X is Fkp-cons. over T+X}\to \ref{item: T+X |- utFkp-LEM + Ukp-DNSC}):$
Let $\vp\in \ut{\F_k}$.
Then $\vp \lor \neg \vp\in \ut{\F_k}$.
Since $\pa \vdash \vp \lor \neg \vp$, we have $\T + X \vdash \vp \lor \neg \vp$ by \eqref{item: PA+X is Fkp-cons. over T+X}.
Let $\psi \in \U_k$.
Then $\wt{\neg \neg \psi}\to \neg\neg \wt{\psi} \in \ut{\F_k}$.
Since $\pa \vdash \wt{\neg \neg \psi}\to \neg\neg \wt{\psi}$, we have $\T + X \vdash \wt{\neg \neg \psi}\to \neg\neg \wt{\psi}$ by \eqref{item: PA+X is Fkp-cons. over T+X}.

\noindent
$(\ref{item: T+X |- utFkp-LEM + Ukp-DNSC} \to \ref{item: T+X |- utFkp-LEM + Ukp-DNEC}):$
It suffices to show $\DNEC{\U_k}$ by using $\LEM{\ut{\F_k}}$ and $\DNSC{\U_k}$.
Since $\ut{\U_k} \subseteq \ut{\F_k}$ and $\LEM{\ut{\U_k}}$ implies $\DNE{\ut{\U_k}}$, by Proposition \ref{prop: decomposition of Gamma-DNEC}, we are done.

\noindent
$(\ref{item: T+X |- utFkp-LEM + Ukp-DNEC} \to \ref{item: T+X |- utSk-LEM + EPk-DNEC}):$ Trivial.

\noindent
$( \ref{item: T+X |- utSk-LEM + EPk-DNEC}\to \ref{item: PA+X is Fkp-cons. over T+X}):$
By Theorem \ref{thm: Fkp cons. <=> (Sk v EPk)-cons.}
and Lemma \ref{lem: an equivalent of (utSk v utEPk)-cons.}, it suffices for \eqref{item: PA+X is Fkp-cons. over T+X} to show that 
$\T +X$ is closed under $\DNER{\EP_k}$ with assumptions of $\Pi_k$ sentences.
Let $\psi\in \ut{\Pi_k}$ and $\vp\in \EP_k$.
Assume $\T +X \vdash  \psi \to \neg \neg \vp$.
Then $\T + X + \psi \vdash \wt{\neg \neg \vp}$.
Since $\T + X$ proves $\DNEC{\EP_k}$ now, we have $\T + X + \psi \vdash \wt{\vp}$, and hence, $\T + X + \psi \vdash \vp$.
Since $\T$ satisfies the deduction theorem, $\T + X \vdash \psi \to \vp$ follows.
\end{proof}

\begin{remark}
$\DNSC{\U_k}$ in Theorem \ref{thm: axiom schemata equivalent to Fkp-conservativity}.\eqref{item: T+X |- utFkp-LEM + Ukp-DNSC} is equivalent over $\ha$ to the closed fragment of $\DNS{\U_k}$:
$$
\neg \neg \forall x \vp \to \forall x \neg\neg \vp,$$
where $\vp\in \U_k$ such that $\FV{\vp}=\{x\}$.
\end{remark}

In the following, we show that $\LEM{\ut{\F_k}}$ and $\DNSC{\U_k}$ in Theorem \ref{thm: axiom schemata equivalent to Fkp-conservativity}.\eqref{item: T+X |- utFkp-LEM + Ukp-DNSC} are independent over $\ha$.

\begin{proposition}
$\ha + \LEM{\ut{\Gamma}} \nvdash \DNSC{(\Pi_1 \lor \Pi_1)}$ for any class $\Gamma$ of $\ha$-formulas. 
\end{proposition}
\begin{proof}
Suppose $\ha + \LEM{\ut{\Gamma}} \vdash \DNSC{(\Pi_1 \lor \Pi_1)}$.
As in the proof of Proposition \ref{prop: PA is not Pi1vPi1-conservative over HA}, let $\Psi(x) \in \Pi_1 \lor \Pi_1$ be the formula \eqref{eq: Psi(x)}.
Since $\ha \vdash \forall x \neg \neg \Psi(x)$, we have
$\ha + \LEM{\ut{\Gamma}} \vdash \neg \neg \forall x \Psi(x)$.
Since the double negation of each instance of $\LEM{\ut{\Gamma}}$ is provable in $\ha$, by (the proof of) \cite[Lemma 4.1]{FK20-1}, we have
$\ha \vdash  \neg \neg \forall x \Psi(x)$.
This is a contradiction as shown in the proof of Proposition \ref{prop: PA is not Pi1vPi1-conservative over HA}.
\end{proof}

\begin{proposition}
\label{prop: HA+GDNSC /|- utS1-LEM}
$\ha +\FDNS \nvdash \LEM{\ut{\Sigma_1}}$ where $\FDNS$ is the axiom scheme of the double-negation-shift  $\forall x (\forall y \neg \neg \vp(x,y)\to \neg \neg \forall y \vp(x, y)) $.
\end{proposition}
\begin{proof}
Let $\vp$ be a sentence in $\Pi_1$ such that $\pa\nvdash \vp$ and $\pa \nvdash \neg \vp$ (e.g. the G\"odel sentence for G\"odel's first incompleteness theorem).
Since each instance of $\FDNS$ is intuitionistically equivalent to a negated sentence (cf. \cite[Remark 2.8]{FK20-1}), by \cite[Theorem 3.1.4 and Lemma 3.1.6]{Tro73}, we have that $\ha + \FDNS$ has the disjunction property.
Suppose $\ha + \FDNS \vdash \vp^\perp \lor \neg \vp^\perp$ (where  $\vp^\perp \in \ut{\Sigma_1}$).
Then, by the disjunction property, we have $\ha + \FDNS \vdash \vp^\perp$ or $\ha + \FDNS \vdash \neg \vp^\perp$, and hence, $\pa \vdash \neg \vp$ or $\pa \vdash \vp$.
This is a contradiction.
\end{proof}

\begin{remark}
By using the disjunction property of $\ha+\FDNS$ as in the proof of Proposition \ref{prop: HA+GDNSC /|- utS1-LEM}, one can extend Proposition \ref{prop: PA is not Pi1vPi1-conservative over HA} to that $\pa$ is not $\left( \Pi_1 \lor \Pi_1\right)$-conservative over $\ha +\FDNS$:
Suppose that $\pa$ is $\left( \Pi_1 \lor \Pi_1\right)$-conservative over $\ha +\FDNS$.
Then, by (the proof of) Theorem \ref{thm: equivalents of (Pk+1 v Pk+1)-cons}, $\ha +\FDNS$ is closed under $\DMLDR{\Sigma_1}$.
Let $\vp$ and $\psi$ be sentences in $\Sigma_1$ such that $\ha$ proves
$$
\vp \lr \exists x \left(\Prf{x}{\vp^\perp} \land \forall y\leq x \neg \Prf{y}{\psi^\perp} \right)
$$
and
$$
\psi \lr \exists y \left(\Prf{y}{\psi^\perp} \land \forall x<y \neg \Prf{x}{\vp^\perp} \right)
$$
where $\Prf{z}{\xi}$ denotes a proof predicate asserting that $z$ is a code of the proof $\xi$ in $ \ha +\FDNS$ (cf. \cite[Chapter 2]{Boolos93}).
Since $\ha \vdash \neg (\vp \land \psi)$, by using $\DMLDR{\Sigma_1}$, we have $\ha +\FDNS \vdash \vp^\perp \lor \psi^\perp $.
Since $\ha +\FDNS$ has the disjunction property, we have that $\ha +\FDNS \vdash \vp^\perp$ or $\ha +\FDNS \vdash \psi^\perp$.
However, in both cases, we have a contradiction by our choice of $\vp$ and $\psi$.
\end{remark}

Next, we show that $\DNEC{\U_k}$, $\DNEC{\EP_k}$  in Theorem \ref{thm: axiom schemata equivalent to Fkp-conservativity} and 
the rule in Lemma \ref{lem: an equivalent of (utSk v utEPk)-cons.} are pairwise equivalent.
\begin{proposition}
\label{prop: equivalents of Ukp-DNEC and EPk-DNEC}
Let 
$\T$ be semi-classical arithmetic
satisfying the deduction theorem
and $X$ be a set of $\ha$-sentences in $\Q_{k}$.
Then the following are pairwise equivalent:
\begin{enumerate}
    \item 
\label{item: T+X |- Ukp-DNEC}
    $\T+X \vdash \DNEC{\U_k};$
    \item
    \label{item: T+X |- EPk-DNEC}
    $\T+X \vdash \DNEC{\EP_k};$
    \item
    \label{item: EPk-DNER with Pk assumptions}
    $\T+X$ is closed under $\DNER{\EP_k}$ with assumptions of sentences in $\Pi_k;$
    \item
    \label{item: Ukp-cons. with Pi_k sentences}
    For any $\psi\in \ut{\Pi_k}$, $\pa +X+\psi$ is $\U_k$-conservative over $\T +X+\psi ;$
    \item
        \label{item: Ukp-DNER with Ukp assumptions}
     $\T+X$ is closed under $\DNER{\U_k}$ with assumptions of sentences in $\U_k;$
\item
    \label{item: Ukp-DNER with any assumptions}
     $\T+X$ is closed under $\DNER{\U_k}$ with assumptions of any sentences.
\end{enumerate}
\end{proposition}
\begin{proof}
The implications $(\ref{item: T+X |- Ukp-DNEC} \to \ref{item: T+X |- EPk-DNEC})$ and $(\ref{item: Ukp-DNER with any assumptions} \to \ref{item: Ukp-DNER with Ukp assumptions})$ are trivial.

\noindent
$(\ref{item: T+X |- EPk-DNEC} \to  \ref{item: EPk-DNER with Pk assumptions}):$ By the proof of $( \ref{item: T+X |- utSk-LEM + EPk-DNEC}\to \ref{item: PA+X is Fkp-cons. over T+X})$ in Theorem \ref{thm: axiom schemata equivalent to Fkp-conservativity}.

\noindent
$(\ref{item: EPk-DNER with Pk assumptions}\to \ref{item: Ukp-cons. with Pi_k sentences}):$
Fix $\psi\in \ut{\Pi_k}$.
Let $\vp\in \U_k$.
Assume $\pa + X +\psi \vdash \vp$.
Since $X \cup \{\psi \} \subseteq \Q_{k}$, by
Theorem \ref{thm: equivalents of Uk+1p-cons},
we have $\T + X +\psi \vdash \vp$.

\noindent
$(\ref{item: Ukp-cons. with Pi_k sentences}\to \ref{item: Ukp-DNER with Ukp assumptions}):$
Assume $\T +X \vdash \psi \to \neg \neg \vp$ where $\psi\in \ut{\U_k}$ and $\vp\in \U_k$.
By Corollary \ref{cor: E_k+1 and Sk+1}.\eqref{item: Uk+1 in cor: E_k+1 and Sk+1}, there exists $\psi'\in \ut{\Sigma_k}$ such that $\ha + \LEM{\Sigma_{k-1}} \vdash \psi' \to \neg \psi $ (cf. Remark \ref{rem: Sk-LEM -> Sk-DNE, Uk-DNS and PkvPk-DNE}) and $\pa \vdash \neg \psi \to \psi'$.
Let $\psi'':\equiv \left( \psi' \right)^\perp$.
By Remark \ref{rem: duals}, we have $\psi'' \in \ut{\Pi_k}$,  $\ha + \LEM{\Sigma_{k-1}} \vdash \neg \neg \psi \to \psi''$
and $\pa \vdash \psi'' \to \psi$.
Then we have  now $\pa + X + \psi'' \vdash \vp$.
By our assumption, $\T + X + \psi'' \vdash \vp$ follows.
Since $T$ satisfies the deduction theorem, we have $\T + X \vdash \psi'' \to \vp$.
On the other hand, by (the proof of) Lemma \ref{lem: SkvPik-cons->Sk-LEM} and our assumption, we have $\T+ X \vdash \LEM{\Sigma_{k-1}}$.
Then $\T+ X \vdash \psi \to \vp$ follows.

\noindent
$(\ref{item: Ukp-DNER with Ukp assumptions}\to \ref{item: T+X |- Ukp-DNEC}):$
Let $\vp\in \U_k$.
Note that $\wt{\neg \neg \vp} \in \ut{\U_k}$.
Since $\T+X+ \wt{\neg \neg \vp} \vdash \neg \neg \vp $, by the deduction theorem, we have $ \T+X \vdash \wt{\neg \neg \vp} \to \neg \neg \vp$.
By our assumption, we have $ \T+X \vdash \wt{\neg \neg \vp} \to  \vp$, and hence, $ \T+X \vdash \wt{\neg \neg \vp} \to \wt{\vp}$.

\noindent
$(\ref{item: T+X |- Ukp-DNEC} \to \ref{item: Ukp-DNER with any assumptions}):$
Assume $\T +X \vdash \psi \to \neg \neg \vp$ where $\psi$ is a sentence and $\vp\in \U_k$. 
Then we have $\T+X +\psi \vdash \wt{\neg \neg \vp}$.
By our assumption, we have  $\T+X +\psi \vdash \wt{\vp}$, and hence, $\T+X +\psi \vdash \vp$.
Since $T$ satisfies the deduction theorem,  $\T+X \vdash  \psi \to \vp$ follows.
\end{proof}

\begin{corollary}\label{cor: U_k-DNEC -> U_k-CONS}
Let $X$ be a set of $\ha$-sentences in $\Q_k$.
Then $\pa +X$ is $\U_k$-conservative over $\ha + X +  \DNEC{\U_k}$.
\end{corollary}

\section{Interrelations between conservation theorems and logical principles}
\label{sec: Summary}
The $\ut{\E_{k+1}}$-conservativity implies both of $\ut{\Sigma_{k+1}}$-conservativity and $\ut{\F_k}$-conservativity.
In what follows, we investigate the relation among them.

\begin{proposition}
Let
$\T$ be semi-classical arithmetic
and $X$ be a set of $\ha$-sentences.
If $\pa +X$ is $\ut{\Sigma_{k+1}}$-conservative over $\T+X$ and $\T+X$ proves $\DNE{(\Pi_k \lor \Pi_k)}$, then $\pa +X$ is $\ut{\E_{k+1}}$-conservative over $\T+X$.
\end{proposition}
\begin{proof}
By Theorem \ref{thm: Ek+1p cons. <=> ESk+1-cons}, it suffices to show $\ut{\ES_{k+1}}$-conservativity instead of the $\ut{\E_{k+1}}$-conservativity.
Let $\vp :\equiv \exists x_1,\dots , x_n\, \psi \in \ut{\ES_{k+1}}$ with $\psi \in \EP_k$.
Assume $\pa + X\vdash \vp$.
By Theorem \ref{thm: PNFT}.\eqref{eq: item for Pi_k in PNFT}, there exists $\psi' \in \Pi_k$ such that $\FV{\psi}=\FV{\psi'}$ and $\ha + \DNE{(\Pi_k \lor \Pi_k)} \vdash \psi'\lr \psi$.
Now we have $\pa + X \vdash  \exists x_1,\dots , x_n\, \psi'$.
Since $\exists x_1,\dots , x_n\, \psi' \in \ut{\Sigma_{k+1}}$, by our first assumption, we have that $\T +X \vdash \exists x_1,\dots , x_n\, \psi'$.
By our second assumption, $\T+ X \vdash \vp$ follows.
\end{proof}

\begin{proposition}
\label{prop: utSk+1-cons => utSk-LEM and Sk-1-LEM}
Let 
$\T$ be a theory containing $\ha$.
If $\pa$ is $\ut{\Sigma_{k+1}}$-conservative over $\T$, then $\T$ proves $\LEM{\ut{\Sigma_k}}$ and also $\LEM{\Sigma_{k-2}}$. 
\end{proposition}
\begin{proof}
Assume that $\pa$ is $\ut{\Sigma_{k+1}}$-conservative over $\T$.
Then  $\pa$ is $\Pi_k$-conservative over $\T$ (cf. Proposition \ref{prop: utcons. => cons. for forall closed classes}), and hence, $\T$ proves $\LEM{\Sigma_{k-2}}$ by (the proof of) Theorem \ref{thm: conservation theorems equivalent to SkLEM}.
Let $\vp\in \ut{\Sigma_k}$.
Then $\vp^\perp \in \ut{\Pi_k}$.
Since $\ut{\Sigma_k}$ and $\ut{\Pi_k}$ can be seen as sub-classes of $\ut{\Sigma_{k+1}}$ and $\ut{\Sigma_{k+1}}$ is closed under $\lor$ (in the sense of \cite[Lemma 4.4]{FK20-1}), one may assume $\vp \lor \vp^\perp \in \ut{\Sigma_{k+1}}$.
Since $\pa \vdash \vp \lor \vp^\perp $, by our assumption, we have $\T \vdash \vp \lor \vp^\perp$, and hence, $\T \vdash \vp \lor \neg \vp$.
\end{proof}

\begin{corollary}
Let
$\T$ be semi-classical arithmetic
satisfying the deduction theorem
and $X$ be a set of $\ha$-sentences in $\Q_{k}$.
If $\pa +X $ is $\ut{\Sigma_{k+1}}$-conservative over $\T+X$ and $\T$ proves $\DNEC{\U_k}$, then $\pa+X$ is $\ut{\F_k}$-conservative over $\T+X$.
\end{corollary}
\begin{proof}
Immediate by Theorem \ref{thm: axiom schemata equivalent to Fkp-conservativity} and Proposition \ref{prop: utSk+1-cons => utSk-LEM and Sk-1-LEM}.
\end{proof}

\begin{remark}
\label{rem: On PkvPk-DNE, Ukp-DNEC, Ukp-cons and utFpk-cons}
By using Theorem \ref{thm: PNFT}.\eqref{eq: item for Pi_k in PNFT}, one can show that $\DNE{(\Pi_k \lor \Pi_k)}$ implies $\DNEC{\U_k}$ in a straightforward way.
On the other hand, $\DNEC{\U_k}$ implies the $\U_k$-conservativity by Corollary \ref{cor: U_k-DNEC -> U_k-CONS}.
In contrast, $\DNE{(\Pi_k \lor \Pi_k)}$ does not imply $\ut{\F_k}$-conservativity since the latter is characterized by $\LEM{\ut{\Sigma_k}} + \DNEC{\U_k}$ (cf. Theorem \ref{thm: axiom schemata equivalent to Fkp-conservativity}) and $\DNE{(\Pi_k \lor \Pi_k)}$
does not imply $\LEM{\ut{\Sigma_k}}$ (see \cite{FINSY20}).
\end{remark}

\begin{remark}
It is straightforward to see that if a theory $\T$ containing $\ha$  proves $\DNEC{(\Pi_k \lor \Pi_k)}$, then $\T$ is closed under $\DNER{(\Pi_k\lor \Pi_k)}$.
Thus $\DNEC{(\Pi_k \lor \Pi_k)}$ implies the $(\Pi_k\lor \Pi_k)$-conservativity (cf. Theorem \ref{thm: equivalents of (Pk+1 v Pk+1)-cons}).
On the other hand, $\DNEC{(\Pi_k \lor \Pi_k)}$ is a fragment of $\DNEC{\U_k}$.
\end{remark}

\begin{proposition}
\label{prop: utSk+1-LEM + Sk-1LEM => utSk+1-cons.}
Let
$X$ be a set of $\ha$-sentences in $\Q_k$.
Then $\pa+X$ is $\ut{\Sigma_{k+1}}$-conservative over $\ha +X+ \DNE{\ut{\Sigma_{k+1}}} +\LEM{\Sigma_{k-1}}$.
\end{proposition}
\begin{proof}
Since $\ha +X+ \DNE{\ut{\Sigma_{k+1}}} +\LEM{\Sigma_{k-1}}$ contains $\LEM{\Sigma_{k-1}}$ and is closed under $\DNER{\ut{\Sigma_{k+1}}}$, by Proposition \ref{prop: utSk+1-cons <=> utSk+1-DNER}, we are done.
\end{proof}

\begin{remark}
Propositions \ref{prop: utSk+1-LEM + Sk-1LEM => utSk+1-cons.} and \ref{prop: utSk+1-cons => utSk-LEM and Sk-1-LEM} reveal that the $\ut{\Sigma_{k+1}}$-conservativity lies between $\LEM{\ut{\Sigma_{k+1}}} +\LEM{\Sigma_{k-1}}$ and $\LEM{\ut{\Sigma_{k}}} +\LEM{\Sigma_{k-2}}$.
This seems to be another view of the status of the $\ut{\Sigma_{k+1}}$-conservativity.
\end{remark}

Our results on the relation between conservation theorems and logical principles are summarized in Figure \ref{fig: Summary} where $\Gamma$-CONS denotes the $\Gamma$-conservativity for class $\Gamma$ of $\ha$-formulas.
Figure \ref{fig: Summary} reveals that the logical principle $\DNEC{\U_k}$, which has been first studied in the current paper (cf. Definition \ref{def: DNEC and DNSC}), is closely related to the conservation theorems.
For the comprehensive information on the arithmetical hierarchy of logical principles including $\LEM{\Sigma_k}$ and $\DNE{(\Pi_{k} \lor \Pi_{k})}$, we refer the reader to \cite{FK20-2}.
For the underivability, we know only that $\LEM{\Sigma_{k-1}}$ does not imply $(\Pi_k \lor \Pi_k)$-CONS (cf. Proposition \ref{prop: PA is not Pi1vPi1-conservative over HA}) and that $\DNE{(\Pi_k\lor \Pi_k)}$ does not imply $\ut{\F_k}$-CONS (cf. Remark \ref{rem: On PkvPk-DNE, Ukp-DNEC, Ukp-cons and utFpk-cons}).
In addition, for $\Gamma \in \{ \Sigma_k, \Pi_k, \Pi_k \lor \Pi_k, \E_k, \F_k, \U_k, \ut{\Sigma_k} \}$, we have characterized $\Gamma$-CONS by some fragment of the double-negation-elimination rule ${\rm DNE\text{-}R}$.
On the other hand, we have not achieved that for $\ut{\E_k}$ and $\ut{\F_k}$.
\begin{figure}[ht]
\centering
\begin{tikzpicture}[scale=0.75]
\node (Sk-1-LEM) at (-6.2,0) {$\LEM{\Sigma_{k-1}}$};
\node (Pk+1-cons) at (-3,0) {$\Pi_{k+1}$-CONS};
\node (Sk-cons) at (0,0) {$\Sigma_{k}$-CONS};
\node (Ek-cons) at (3,0) {$\E_{k}$-CONS};
\node (Fk-1-cons) at (6,0) {$\F_{k-1}$-CONS};
\node (PkvPk-cons) at (0,1.1) {$(\Pi_k \lor \Pi_k)$-CONS};
\node (PkvPk-DNEC) at (-6,2.2) {$\DNEC{(\Pi_k \lor \Pi_k)}$};
\node (Uk-cons) at (0,2.2) {$\U_{k}$-CONS};
\node (Uk-DNEC) at (0,3.25) {$\DNEC{\U_k}$};
\node (PkvPk-DNE + utSk-LEM) at (-6, 5.2) {$\DNE{(\Pi_k\lor\Pi_k)} + \LEM{\ut{\Sigma_{k}}}$};
\node (PkvPk-DNE) at (-6, 3.8) {$\DNE{(\Pi_k\lor\Pi_k)}$};
\node (utFk-cons) at (0,4.4) {$\ut{\F_{k}}$-CONS};
\node (Uk-DNEC and utS-LEM) at (4,4.4) {$\DNEC{\U_k}+ \LEM{\ut{\Sigma_{k}}}$};
\node (Uk-DNEC and utSk+1-cons) at (0,5.6) {$\DNEC{\U_k}$ \& $\ut{\Sigma_{k+1}}$-CONS};
\node (utEk+1-cons) at (0,6.8) {$\ut{\E_{k+1}}$-CONS};
\node (PkvPk-DNE and utSk+1-cons) at (0,8) {$\DNE{(\Pi_k \lor \Pi_k)}$ \& $\ut{\Sigma_{k+1}}$-CONS};
\node (Sk-LEM) at (-6.3,9.2) {$\LEM{\Sigma_{k}}$};
\node (*1) at (-2.7, 7.7) {};
\node (Pk+2-cons) at (-3.3, 9.2) {$\Pi_{k+2}$-CONS};
\node (Sk+1-cons) at (0,9.2) {$\Sigma_{k+1}$-CONS};
\node (Ek+1-cons) at (3.2,9.2) {$\E_{k+1}$-CONS};
\node (Fk-cons) at (6.2,9.2) {$\F_{k}$-CONS};

\draw [<->] (Sk-1-LEM)--(Pk+1-cons);
\draw [<->] (Sk-cons)--(Pk+1-cons);
\draw [<->] (Sk-cons)--(Ek-cons);
\draw [<->] (Fk-1-cons)--(Ek-cons);
\draw [<-] (Sk-cons)--(PkvPk-cons);
\draw [<-] (PkvPk-cons)--(Uk-cons);
\draw [<-] (Uk-cons)--(Uk-DNEC);
\draw [<-] (Uk-DNEC)--(utFk-cons);
\draw [<-] (utFk-cons)--(Uk-DNEC and utSk+1-cons);
\draw [<-] (Uk-DNEC and utSk+1-cons)--(utEk+1-cons);
\draw [<-] (utEk+1-cons)--(PkvPk-DNE and utSk+1-cons);
\draw [<-] (PkvPk-DNE and utSk+1-cons)--(Sk+1-cons);
\draw [<->] (Sk-LEM)--(Pk+2-cons);
\draw [<->] (Sk+1-cons)--(Pk+2-cons);
\draw [<->] (Sk+1-cons)--(Ek+1-cons);
\draw [<->] (Fk-cons)--(Ek+1-cons);
\draw [<-] (PkvPk-DNE + utSk-LEM)--(*1);
\draw [->] (PkvPk-DNE + utSk-LEM)--(utFk-cons);
\draw [->] (PkvPk-DNE + utSk-LEM)--(PkvPk-DNE);
\draw [<-] (Uk-DNEC)--(PkvPk-DNE);
\draw [<-] (PkvPk-DNEC)--(Uk-DNEC);
\draw [->] (PkvPk-DNEC)--(PkvPk-cons);
\draw [<->] (Uk-DNEC and utS-LEM)--(utFk-cons);

\end{tikzpicture}
 \caption{Conservation theorems in the arithmetical hierarchy of logical principles}
    \label{fig: Summary}
\end{figure}
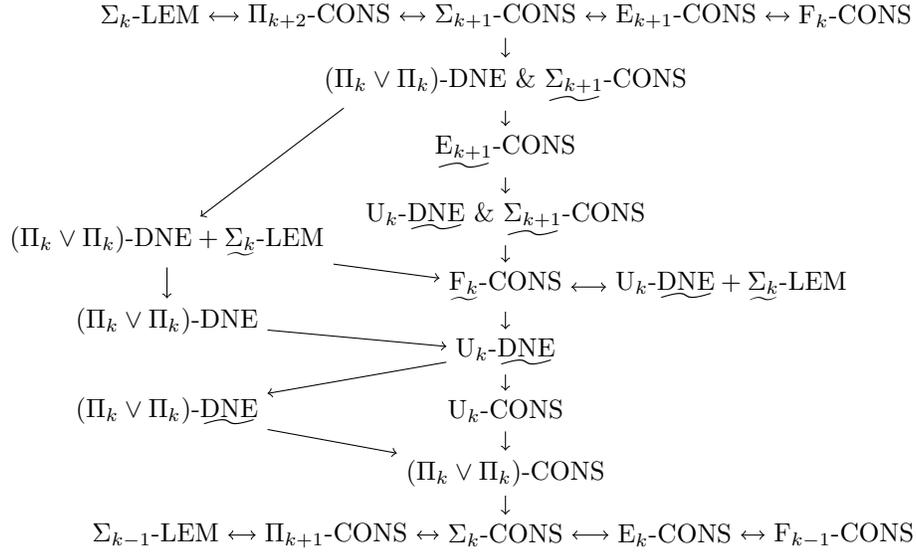

\section{Appendix: A relativized soundness theorem of the Friedman A-translation for $\ha+\LEM{\Sigma_k}$}
\label{section: Soundness of A-Translation}
We provide a detailed proof of a relativized soundness theorem of the Friedman A-translation \cite{Fri78} for $\ha+\LEM{\Sigma_k}$ (see Theorem \ref{A relativized soundness theorem of *-translation}).
In fact, this result was suggested already in \cite[Section 4.4]{HN02} and the detailed proof for $k=1$ can be found
in \cite[Lemma 3.1]{KS14}.
The authors, however, couldn't find the proof for arbitrary natural number $k$ anywhere, which is the reason why we present the detailed proof here.
For the relativized soundness theorem,
we use a variant of of Lemma \ref{Kurahashi Lemma: A* <-> A v *} with respect to the Friedman A-translation.

We first recall the definition of the Friedman A-translation.
In this section, we use symbol $\PH$ for place holder instead of $\ph$ in the previous sections.
\begin{definition}[A-translation \cite{Fri78}]
\label{def: A-translation}
For a $\ha$-formula $\vp$, we define $\vp^{\PH} $ as a formula obtained from $\vp$ by replacing all the prime formulas $\vp_{\rm p}$ in $\vp$ with $\vp_{\rm p} \lor {\PH} $ (of course, $\vp^{\PH} $ is officially defined by induction on the logical structure of $\vp$).
In particular, $\perp^{\PH} :\equiv  \left(\perp \lor \, {\PH} \right)$, which is equivalent to  ${\PH} $ over $\ha^{\PH} $ ($\ha$ in the language with a place holder $\PH$).
In what follows, $\neg_{\PH}\, \vp$ denotes $\vp \to {\PH} $.
Note that $\FV{\vp} =\FV{\vp^{\PH}}$ for all $\ha$-formulas $\vp$.
\end{definition}

The following is a variant of Lemma \ref{Kurahashi Lemma: A* <-> A v *} with respect to the Friedman A-translation.
\begin{lemma}
\label{lem: Kurahashi Original Lemma: A* <-> A v *}
For a formula $\vp$ of $\ha$,
the following hold$:$
\begin{enumerate}
    \item
    \label{item: A v * for Pi_{k+1} for KOL}
    If $\vp\in \Pi_{k}$, $\ha^{\PH} + \LEM{\Sigma_{k}} \vdash \vp^{\PH} \lr \vp \lor {\PH} ;$
    \item
    \label{item: A v * for Sigma_{k+1} for KOL}
    If $\vp\in \Sigma_{k}$, $\ha^{\PH} + \LEM{\Sigma_{k-1}} \vdash \vp^{\PH} \lr \vp \lor {\PH}.$
\end{enumerate}
\end{lemma}
\begin{proof}
By simultaneous induction on $k$.
The base case is verified by a routine inspection.
Assume items \ref{item: A v * for Pi_{k+1} for KOL} and \ref{item: A v * for Sigma_{k+1} for KOL} for $k$ to show those for $k+1$.
The first item for $k+1$ is shown by using the second item for $k$ as in the proof of Lemma \ref{Kurahashi Lemma: A* <-> A v *}.
For the second item, let $\vp:\equiv \exists x \vp_1$ where $\vp_1 \in \Pi_k$.
Then we have that  $\ha +  \LEM{\Sigma_{k}}$ proves
$$
\vp^{\PH} \equiv \exists x \left( {\vp_1}^{\PH} \right) \underset{\text{[I.H.] }\LEM{\Sigma_{k}}}{\llr} 
 \exists x(\vp_1 \lor \PH) \llr \vp \lor \PH.
$$
\end{proof}

\begin{theorem}
\label{A relativized soundness theorem of *-translation}
If $\ha + \LEM{\Sigma_k} \vdash \vp$, then $\ha^{\PH} + \LEM{\Sigma_k} \vdash \vp^{\PH}$.
\end{theorem}
\begin{proof}
By induction on the length of the proof of $\vp$ in $\ha + \LEM{\Sigma_k}$.
By (the proof of) \cite[Lemma 2]{Fri78}, it suffices to show $\ha^{\PH} +\LEM{\Sigma_k} \vdash  \vp^{\PH}$ for each instance $\vp$ of $\LEM{\Sigma_k}$.
Fix $\vp: \equiv \exists x \vp_1 \lor \neg \exists x \vp_1$ with $\vp_1 \in \Pi_{k-1}$.
By Lemma \ref{lem: Kurahashi Original Lemma: A* <-> A v *}.\eqref{item: A v * for Pi_{k+1} for KOL}, $\ha^{\PH} + \LEM{\Sigma_{k-1}}$ proves
$$\begin{array}{ccl}
\vp^{\PH} &\llr & \exists x \left( {\vp_1}^{\PH}\right) \lor \neg_{\PH} \exists x  \left( {\vp_1}^{\PH} \right)
\\
&\underset{\LEM{\Sigma_{k-1}}}{\llr}  &
\exists x \left( \vp_1 \lor \PH \right) \lor \neg_{\PH} \exists x  \left( \vp_1\lor \PH \right)\\
& \llr & \exists x \left( \vp_1 \lor \PH \right) \lor \neg_{\PH} \exists x \vp_1,
\end{array}$$
which is derived from $\exists x \vp_1 \lor \neg \exists x \vp_1$ over $\ha^{\PH}$.
Thus $\ha^{\PH} +\LEM{\Sigma_k} $ proves $\vp^{\PH}$.
\end{proof}

By the relativized soundness theorem of the Friedman A-translation combined with the usual negative translation, one can show Proposition \ref{prop: PA is Pi_k+2-cons. over HA + Sigma_k-LEM}
as follows:
\begin{proof}[Proof Sketch of Proposition \ref{prop: PA is Pi_k+2-cons. over HA + Sigma_k-LEM}]
Assume $\pa \vdash \forall x  \exists y \vp$ where $\vp\in \Pi_{k}$.
By using Kuroda's negative translation (cf. \cite[Proposition 6.4]{FK20-1}), we have $\ha \vdash \forall x \neg\neg \exists y \vp_N$ where $\vp_N$ is defined as in \cite[Definition 6.1]{FK20-1}.
Since $\ha+\LEM{\Sigma_{k-1}} $ proves $\DNE{\Sigma_{k-1}}$, we have $\ha+\LEM{\Sigma_{k-1}} \vdash \neg \neg \exists y \vp$ (cf. \cite[Lemma 6.5.(2)]{FK20-1}).
By Theorem \ref{A relativized soundness theorem of *-translation}, we have  $\ha^{\PH}+\LEM{\Sigma_{k-1}} \vdash \neg_{\PH} \neg_{\PH} \exists y \vp^{\PH}$, and hence,
$\ha^{\PH}+\LEM{\Sigma_k} \vdash \neg_{\PH} \neg_{\PH} \exists y \vp$ by Lemma \ref{lem: Kurahashi Original Lemma: A* <-> A v *}.\eqref{item: A v * for Pi_{k+1} for KOL}.
By substituting $\PH$ with $\exists y \vp$ (cf. Lemma \ref{lem: Substitution}), we have that $\ha +\LEM{\Sigma_k} $ proves $ \exists y \vp$, and hence, $\forall x \exists y \vp$.
\end{proof}

The proof of \cite[Theorem 3.5.5]{ConstMathI} (due to Visser) shows that any theory $\T$ which contains $\ha$ and is sound for the Friedman A-translation is closed under the independence-of-premise rule:
\begin{center}
$\T \vdash \neg \vp \to \exists x \psi $ implies $\T \vdash \exists x \left( \neg \vp\to\ \psi \right)$
\end{center}
where $x\notin \FV{\neg \vp}$.
Then, by using Theorem \ref{A relativized soundness theorem of *-translation}, we also have the following:
\begin{theorem}
$\ha + \LEM{\Sigma_{k}}$ is closed under the independence-of-premise rule.
\end{theorem}

\section*{Acknowledgements}
The first author was supported by JSPS KAKENHI Grant Numbers JP19J01239 and JP20K14354, and the second author by JP19K14586.

\bibliographystyle{abbrv}
\bibliography{bibliography.bib}

\begin{thebibliography}{10}

\bibitem{ABHK04}
Y.~Akama, S.~Berardi, S.~Hayashi, and U.~Kohlenbach.
\newblock An arithmetical hierarchy of the law of excluded middle and related
  principles.
\newblock In {\em Proceedings of the 19th Annual IEEE Symposium on Logic in
  Computer Science (LICS'04)}, pages 192--301. 2004.

\bibitem{Boolos93}
G.~{Boolos}.
\newblock {\em {The logic of provability}}.
\newblock Cambridge: Cambridge University Press, 1993.

\bibitem{Fri78}
H.~Friedman.
\newblock Classically and intuitionistically provably recursive functions.
\newblock In G.~H. {M\"u}ller and D.~S. Scott, editors, {\em Higher Set
  Theory}, pages 21--27, Berlin, Heidelberg, 1978. Springer Berlin Heidelberg.

\bibitem{FINSY20}
M.~Fujiwara, H.~Ishihara, T.~Nemoto, N.-Y. Suzuki, and K.~Yokoyama.
\newblock Extended frames and separations of logical principles.
\newblock \url{https://researchmap.jp/makotofujiwara/misc/30348506}, 2021.
\newblock submitted.

\bibitem{FK20-2}
M.~Fujiwara and T.~Kurahashi.
\newblock Refining the arithmetical hierarchy of classical principles.
\newblock \url{https://arxiv.org/abs/2010.11527}, 2020.
\newblock submitted.

\bibitem{FK20-1}
M.~Fujiwara and T.~Kurahashi.
\newblock Prenex normal form theorems in semi-classical arithmetic.
\newblock {\em Journal of Symbolic Logic}, 86(3):1124--1153, 2021.

\bibitem{HN02}
S.~Hayashi and M.~Nakata.
\newblock Towards limit computable mathematics.
\newblock In P.~Callaghan, Z.~Luo, J.~McKinna, R.~Pollack, and R.~Pollack,
  editors, {\em Types for Proofs and Programs}, pages 125--144, Berlin,
  Heidelberg, 2002. Springer Berlin Heidelberg.

\bibitem{Ishi00}
H.~Ishihara.
\newblock A note on the {G}\"odel-{G}entzen translation.
\newblock {\em Mathematical Logic Quarterly}, 46(1):135--137, 2000.

\bibitem{Ishi12}
H.~Ishihara.
\newblock Some conservative extension results on classical and intuitionistic
  sequent calculi.
\newblock In U.~Berger, H.~Diener, P.~Schuster, and S.~Monika, editors, {\em
  Logic, Construction, Computation}, Ontos mathematical logic, pages 289 --
  304. De Gruyter, Berlin, Boston, 2012.

\bibitem{KP88}
R.~Kaye, J.~Paris, and C.~Dimitracopoulos.
\newblock On parameter free induction schemas.
\newblock {\em Journal of Symbolic Logic}, 53(4):1082–1097, 1988.

\bibitem{KS14}
U.~Kohlenbach and P.~Safarik.
\newblock Fluctuations, effective learnability and metastability in analysis.
\newblock {\em Annals of Pure and Applied Logic}, 165(1):266 -- 304, 2014.

\bibitem{Tro73}
A.~S. Troelstra, editor.
\newblock {\em Metamathematical investigation of intuitionistic arithmetic and
  analysis}, volume 344 of {\em Lecture Notes in Mathematics}.
\newblock Springer-Verlag, Berlin, New York, 1973.

\bibitem{ConstMathI}
A.~S. Troelstra and D.~van Dalen.
\newblock {\em Constructivism in mathematics, An introduction, {V}ol. {I}},
  volume 121 of {\em Studies in Logic and the Foundations of Mathematics}.
\newblock North Holland, Amsterdam, 1988.

\bibitem{vD13}
D.~van Dalen.
\newblock {\em Logic and Structure}.
\newblock Universitext. Springer-Verlag London, fifth edition, 2013.

\end{thebibliography}
\end{document}